\documentclass[reqno]{amsart}

\usepackage{xcolor,enumerate}       
 \definecolor{darkgreen}{rgb}{0,0.5,0}

\usepackage[T1]{fontenc}    
\usepackage{url}            
\usepackage{booktabs}       
\usepackage{amsfonts}       
\usepackage{nicefrac}       
\usepackage{microtype}      
\usepackage{amssymb,latexsym,amsfonts,amsmath}
\usepackage{graphicx,xspace,bm} 
\usepackage{subcaption}
\usepackage{color}
\usepackage{mathabx}
\usepackage{algorithmic}
\usepackage{algorithm}

\usepackage{tikz}
\usetikzlibrary{external}
\usetikzlibrary{decorations.shapes,arrows,calc,decorations.markings,shapes,decorations.pathreplacing,shapes.arrows}
\usetikzlibrary{positioning}
\tikzset{main node/.style={circle,fill=blue!20,draw,inner sep=1pt},}

\definecolor{BBlue}{cmyk}{.98,0.10,0,.25}
\definecolor{LightBlue}{rgb}{0.2,0.6,1}
\definecolor{DarkBlue}{rgb}{0,0,.5}
\definecolor{DarkGreen}{rgb}{0,0.4,.2}
\definecolor{LightGreen}{rgb}{0.6,1,.4}
\usepackage[colorlinks=true,allcolors=black, linktocpage=true]{hyperref}
\definecolor{RRed}{cmyk}{0,0.6,0.4,0}
\definecolor{GGrey}{rgb}{1,0.97,0.97} 

\newtheorem{theorem}{Theorem}[section]

\newtheorem{problem}[theorem]{Problem}
\newtheorem{proposition}[theorem]{Proposition}
\newtheorem{corollary}[theorem]{Corollary}

\newtheorem{definition}[theorem]{Definition}

\newtheorem{remark}[theorem]{Remark}
\numberwithin{equation}{section}
\newtheorem{assumption}[theorem]{Assumption}

\DeclareMathOperator*{\argmin}{arg\,min}

\newcommand{\real}{{\mathbb{R}}}

\newcommand{\x}{{\mathbf{x}}}

\newcommand\subscr[2]{#1_{\textup{#2}}}

\newcommand{\oprocendsymbol}{\hbox{$\bullet$}}
\newcommand{\oprocend}{\relax\ifmmode\else\unskip\hfill\fi\oprocendsymbol}

\newcommand{\hodclf}{\textsf{g-dclf}}
\newcommand{\adc}{\textsf{adc}}
\newcommand{\boldnu}{\boldsymbol{\nu}}
\newcommand{\boldu}{\mathbf{u}}
\DeclareMathOperator{\vect}{vec}
\newcommand\ld{\subscr{\ell}{decr}}
\newcommand{\interior}{\operatorname{int}}
\newcommand{\inistatetimev}{x^{k_0}}
\newcommand{\initimetimev}{k_0}

\makeatletter
\renewcommand\subsection{\@startsection{subsection}{2}%
  \z@{-.5\linespacing\@plus-.7\linespacing}{.5\linespacing}%
  {\normalfont\scshape}}
\renewcommand\subsubsection{\@startsection{subsubsection}{3}%
  \z@{.5\linespacing\@plus.7\linespacing}{-.5em}%
  {\normalfont\scshape}}
\makeatother

\begin{document}

\title[]{Flexible-step Model Predictive Control\\based on generalized Lyapunov functions}

\begin{abstract}
We present a novel nonlinear model predictive control (MPC) scheme with relaxed stability criteria, based on the idea of generalized discrete-time control Lyapunov functions. These functions need to satisfy an average descent over a finite window of time, rather than a descent at every time step. 
One feature of this scheme is that it allows for implementing a flexible number of control inputs in each iteration, in a computationally attractive manner, while guaranteeing recursive feasibility and stability. 
The benefits of our flexible-step implementation are also demonstrated in an application to  nonholonomic systems, where the one-step standard implementation may suffer from lack of asymptotic convergence.
\end{abstract}


\author[Annika F\"urnsinn]{Annika F\"urnsinn}
\address{Department of Mathematics and Statistics\\ Queen's University, Kingston, ON K7L 3N6, Canada}
\email{19af7@queensu.ca}
\author[Christian Ebenbauer ]{Christian Ebenbauer }
\address{Chair of Intelligent Control Systems\\
RWTH Aachen University, 52062 Aachen, Germany}
\email{christian.ebenbauer@ic.rwth-aachen.de}
\author[Bahman Gharesifard ]{Bahman Gharesifard }
\address{Department of Electrical and Computer Engineering\\
University of California at Los Angeles,
Los Angeles, CA 90095}
\email{gharesifard@ucla.edu}

\maketitle

\section{Introduction}\label{sec:intro}
Model predictive control (MPC) is a method to approximate the solution of an infinite-horizon optimal control problem by considering a sequence of finite-horizon optimal control problems implemented in a receding horizon fashion. In the standard discrete-time MPC scheme, the finite-horizon optimal control problem is solved at each time step with the current state as the initial state, and the first element of the derived optimal control sequence is implemented.

Since stability of MPC schemes does not follow from the associated infinite-horizon optimal control problem, several approaches have been proposed to guarantee stability~\cite{DM-JR-CR-PS:00, DM:14,JR-DM-MD:17, LG-JP:17}. In the standard MPC scheme, a terminal cost and a terminal constraint are added to the optimal control problem. 
By considering the optimal value function as the Lyapunov function, a strictly decreasing sequence of Lyapunov function values is obtained, which implies convergence of the state to zero. In this paper, we replace the \emph{strict descent} of the Lyapunov function values with an \emph{average descent} while preserving stability. We propose a new MPC scheme which enables the implementation of a flexible number of elements of the optimal input sequence. Here, with flexible number we mean that in each iteration the number of implemented elements of the corresponding optimal input sequence is not constant (e.g. one) but instead flexible (variable). The flexible-step implementation is enabled by the novel notion of generalized discrete-time control Lyapunov functions (\hodclf), which depend on the state \emph{and} the control. The existence of such a \hodclf~for a control system makes it possible to add a \emph{feasible} average decrease constraint on the Lyapunov function values to the finite-horizon optimal control problem of our MPC scheme. The average descent in turn guarantees that at least one predicted Lyapunov function value along the trajectory of the system is less than the Lyapunov function value at the beginning of the prediction. In the proposed MPC scheme, we implement the optimal input sequence until a descent occurs. Consequently, this part of our scheme gives rise to a flexible-step implementation. 

\subsection{Related Work}
An important objective in many MPC schemes is stability; the three most common approaches can be described as follows. First, there are schemes that rely on terminal constraints and/or costs, e.g. standard MPC~\cite{DM-JR-CR-PS:00,DM:14, JR-DM-MD:17, LG-JP:17}. One issue that this approach faces is that with the typical choice of a quadratic cost function, MPC might not be stabilizing~\cite{MM-KW:17}. This can be resolved by designing a non-quadratic cost function based on homogeneous approximation~\cite{JMC-LG-KW:20}. Second, terminal ingredients are avoided altogether through a sufficiently long prediction horizon~\cite{MA-GB:95, AJ-JP-JH:01, LG-JP:17}. Finally, there are schemes that directly incorporate a Lyapunov function decay constraint in the optimization problem, often referred to as Lyapunov-based~\cite{PM-NE-PC:05}.

There are several MPC approaches that aim to relax the one-step standard implementation, one being the so-called ``finite-step'' approaches~\cite{NN-RG-LG-FW:20,ML-VS:15,JH-ML-RT:17} the standard MPC assumptions are relaxed. 
In~\cite{JH-ML-RT:17}, a stabilizing tube-based MPC for linear parameter-varying systems is presented. The usual invariant set used in standard MPC is replaced by finite-step contractive sets, which are easier to compute. In the latter sets, the state of a control system is allowed to leave the set as long as it returns after a finite number of time steps.
Based on~\cite{SdOK-MM:00}, the work in~\cite{NN-RG-LG-FW:20} generalizes control Lyapunov functions by finite-step control Lyapunov functions. Such functions are allowed to increase along the trajectory of a system as long as at the end of the prediction horizon (after a finite number of steps) they decay compared to the beginning of the prediction. In the finite-step control Lyapunov function-based MPC scheme, the considered finite-step control Lyapunov function is the objective function and is incorporated as a constraint. The exact same finite number of steps (greater than one) is implemented in each iteration of the scheme.
In \cite{ML-VS:15}, finite-step control Lyapunov functions are used alongside finite–step contractive sets. More precisely, the finite-step Lyapunov functions are used in the terminal cost of the MPC optimization problem. As in~\cite{NN-RG-LG-FW:20}, the Lyapunov function only depends on the state.

Our work, however, belongs to the category of so-called ``flexible-step'' approaches~\cite{MA:17,THY-EP:93,KW-MM-GM-RG-JP:17}, where a flexible number of elements of the optimal input sequence is implemented in the MPC scheme. A \emph{variable} prediction horizon~\cite{HM-DM:93,HM:97} is employed in contractive MPC~\cite{THY-EP:93}. In the optimal control problem of this setting, the prediction horizon is a decision variable. Hence, solving the MPC scheme yields the optimal control sequence and the optimal prediction horizon. In each iteration, the optimal control sequence is then implemented over the whole optimal prediction horizon. In the discrete-time setting, optimizing over the prediction horizon results in an undesirable integer decision variable. In~\cite{MA:17}, a contraction-based MPC scheme is proposed, which avoids the integer optimization for the price of solving two optimization problems in each iteration. Based on a finite-step Lyapunov inequality, a flexible-step implementation is accomplished in~\cite{KW-MM-GM-RG-JP:17} to reduce the computational burden of the proposed MPC scheme. Although it serves a different purpose, we want to mention event-triggered MPC~\cite{AE-DD-KK:11} as a way of achieving a flexible-step implementation.

\subsection{Contribution}
In contrast to the former MPC schemes, the proposed scheme in this work achieves stability, while allowing the implementation of a flexible number of steps, 
thanks to the average decrease constraint, without the above-mentioned computational drawbacks.

One main technical contribution of this paper is the introduction of generalized discrete-time control Lyapunov functions, which cover classical Lyapunov functions but also the control-dependent objective function. 
Notably, the therewith proposed flexible-step approach includes many known MPC schemes, such as terminal ingredient-based MPC and contraction-based MPC~\cite{DM-JR-CR-PS:00, DM:14}. Furthermore, we derive theoretical results about recursive feasibility and stability for our proposed scheme. We demonstrate the functionality of our method by utilizing it to study the stabilization of nonholonomic systems, which are typically a bottleneck for standard MPC schemes. Several simulation studies confirm the anticipated benefits of the flexibility coming from the \hodclf s.

\subsection{Organization}
The paper is organized as follows. We explain our notation in Section~\ref{sec:prelim}. In Section~\ref{sec:novel_scheme}, we present the main results of this paper, from which we derive a novel implementation scheme for MPC. The proofs of the main results can be found in Section~\ref{sec:stability_proofs}. In Section~\ref{sec:application_nonholo_sys}, we apply the derived scheme to nonholonomic systems. We present the numerical results and analyze these in detail. Finally, we draw our conclusions and discuss some future directions in Section~\ref{sec:conclusion}. 

\subsection{Notation}\label{sec:prelim}
We denote the set of non-negative (positive) integers by $\mathbb{N}$ ($\mathbb{N}_{>0}$), the set of (non-negative) reals by $\real$ ($\real_{\geq 0}$), and the interior of a subset $S \subseteq \mathbb{R}^n$ by $\interior(S)$. The set of real matrices of the size $p\times q$ is denoted by $\real^{p,q}$. We make use of boldface when considering a finite sequence of vectors, e.g. $\mathbf{u}= [u_{0},\dots,u_{N-1}] \in \real^{p,N}$, we refer to $j$th component by $u_j$ and to the subsequence going from component $i$ to $j$ by $\mathbf{u}_{[i:j]}$. Similarly, $\mathbf{U}_{[0:N-1]} \subseteq \mathbb{R}^{p,N}$ denotes a set of (ordered) sequences of vectors with components zero to $N-1$. When such a set depends on the initial state $x$, it is expressed by $\mathbf{U}_{[0:N-1]}(x)$. As usual, $\boldu^{*}$ denotes the solution to the optimal control problem solved in the iteration of an MPC scheme, the symbol $\boldu^{*-}$ denotes the optimal control sequence of the \emph{previous} iteration. For later use, we recall that a function $V: \mathbb{R}^n\times\mathbb{R}^{p, q} \to \mathbb{R}$ is called positive definite if $V(x,\boldu) = 0$ is equivalent to $ (x,\boldu) = (0,\mathbf{0})$ and $ V(x,\boldu) > 0$ for all $(x,\boldu) \in \mathbb{R}^n\times\mathbb{R}^{p, q}\setminus \{(0, \mathbf{0})\}$. Furthermore, a function $\alpha: \mathbb{R}^n\times\mathbb{R}^{p, q} \to \mathbb{R}$ is called radially unbounded if $\|(x,\boldu)\| \to \infty$ implies $\alpha(x,\boldu) \to \infty$, where $\|\cdot\|$ is the Euclidean norm. Note that $x$ is a vector and $\boldu$ is a matrix here, so with $\|(x,\boldu)\|$ we implicitly refer to the norm of $[x^T,\vect(\boldu)^T]$, where $\vect(\boldu)$ is the usual vectorization of a matrix into a column vector.

In the setting of MPC, it is important to distinguish between predictions and the actual implementations at a given time index $k \in \mathbb{N}$.
In particular, we use $x^k, u^k$ to refer to \emph{predictions}, whereas we use $x(k), u(k)$ to refer to the \emph{actual} states and \emph{implemented} inputs, respectively.

\section{Main results}\label{sec:novel_scheme}
Our main objective in this work is to propose a new implementation scheme with some key advantages that we discuss shortly. In order to introduce this scheme, we require a few notions related to stability of control systems. Throughout this paper, we consider nonlinear discrete-time control systems of the form
\begin{align}\label{eqn:control_system}
    x^{k+1} &= f(x^k,u^k),
\end{align} 
where $ k \in \mathbb{N} $ denotes the time index, $x^k \in X \subseteq \mathbb{R}^n$ is the state with the initial state $x^0 \in X$ and $u^k \in U \subseteq \mathbb{R}^p$ is an input. The state and input constraints satisfy $0 \in \interior(X)$, $0 \in \interior(U)$ and $f: X \times U \to \mathbb{R}^n$ is continuous with $f(0,0)=0$. No modeling uncertainties or disturbances are considered in this work. For later use, we recall the asymptotic stability of a time-varying system~\cite{WH-VC:08, AM-LH-DL:08}. Let $k_0 \in \mathbb{N}$ and consider 
\begin{align}
\label{eqn:time-var-control-system}
    x^{k+1} = F(k,x^k), \quad k \geq k_0,
\end{align}
where $F: [k_0, \infty)\cap \mathbb{N}\times X \to \mathbb{R}^n$ and $F(k,0)=0$ for all $k \geq k_0$.
The origin is an asymptotically stable equilibrium of~\eqref{eqn:time-var-control-system} (with region of attraction $X$) if
\begin{enumerate}
    \item for all $\varepsilon>0$ and any $k_0\in \mathbb{N}$, there exists $\delta=\delta(\varepsilon,k_0)>0$ such that $\inistatetimev \in X$ with $0 <\|\inistatetimev\|<\delta$ implies $\|x^k\| < \varepsilon$ for all $k \geq k_0$;
    \item for any $k_0\in \mathbb{N}$ and $\inistatetimev \in X$, it holds $\lim_{k \to \infty}x^k=~0$.
\end{enumerate}  

\subsection{Generalized control Lyapunov function}\label{sec:notions_of_stab}
We now introduce \hodclf s that are core to our proposed MPC scheme. 
We present our stability results, which are the building blocks of our proposed scheme, postponing the proofs to Section~\ref{sec:stability_proofs}. 

The classical condition for a Lyapunov function $V$ is $V(x^{k+1}) < V(x^k)$. It requires that $V$ strictly decreases along the trajectories of a control system and results in asymptotic stability of the origin. \emph{Higher-order} Lyapunov functions (continuous time) were originally introduced in~\cite{AB:69} and therewith replaced the classical Lyapunov condition. In~\cite{AA-PP:08} the classical condition was relaxed by introducing so-called \emph{non-monotonic} Lyapunov functions, which are characterized by the following condition in discrete-time
\begin{align}\label{eqn:non-mono-lyap}
    \tau (V(x^{k+2}) - V(x^{k})) + (V(x^{k+1}) - V(x^{k})) < 0, 
\end{align}
for $\tau \geq 0$. Note that for $\tau=1$,~\eqref{eqn:non-mono-lyap} imposes that the average of the next two future values of the Lyapunov function should be less than the current value of the Lyapunov function. 
We now generalize these ideas.
\begin{definition}[Set of Feasible Controls]
For $x\in X$, we define a set of feasible controls of length $N \in \mathbb{N}_{>0}$ as 
\begin{align}\label{eqn:setfeasc}
    &\mathbf{U}_{[0:N-1]}(x):=\{\boldu_{[0:N-1]} \in \mathbb{R}^{p,N}: \text{ with $x^0=x$},u_j \in U, x^{j+1} = f(x^j,u_j) \in X,\\
&\hspace{5cm} j = 0, \dots, N-1\}.\nonumber
\end{align} 
An infinite sequence of control inputs is called feasible when equation~\eqref{eqn:setfeasc} is fulfilled for all times $j \in \mathbb{N}$. 
\end{definition}
\begin{definition}[\hodclf]\label{def:ho-dclf}
Consider the control system~\eqref{eqn:control_system}. Let $m \in \mathbb{N}_{>0}$ and $q \in \mathbb{N}$. We call $V: \mathbb{R}^n \times \mathbb{R}^{p, q} \to \mathbb{R} $ a generalized discrete-time control Lyapunov function of order $m$ (\hodclf) for system~\eqref{eqn:control_system} if $V$ is continuous, positive definite and additionally:

\underline{If $q = 0$:} 
\begin{enumerate}[i)]
\setcounter{enumi}{0}
    \item there exists a continuous, radially unbounded and positive definite function ${\alpha: \mathbb{R}^n \to \mathbb{R}}$ such that
    for any $x^0 \in$ $\mathbb{R}^n $ we have
\begin{align*}
    V(x^0) - \alpha(x^0) \geq 0;
\end{align*}
\item for any $x^0 \in \mathbb{R}^n$ there exists $\boldnu_{[0:m-1]} \in \mathbf{U}_{[0:m-1]}(x^0)$, which steers $x^0$ to some $x^{m}$, such that
\begin{align}\label{eqn:adc_q0}
    &\frac{\sigma_m V(x^{m}) + \dots +\sigma_1 V(x^1)}{m} -V(x^0) \leq -\alpha(x^0),
\end{align}
where $\sigma_m, \ldots, \sigma_1  \in \real_{\geq 0}$ and 
\begin{equation}\label{eqn:sigma_ave}
    \frac{\sigma_m + \sigma_{m-1} + \ldots + \sigma_1}{m} -1 \geq 0.
\end{equation}
\end{enumerate} 

\underline{If $q \neq 0$:}
\begin{enumerate}[i')]
\setcounter{enumi}{0}
    \item there exists a continuous, radially unbounded and positive definite function ${\alpha: \mathbb{R}^n \times \mathbb{R}^{p, q}\to \mathbb{R}}$ such that for any $(x^0,\boldu_{[0:q-1]}) \in$ ${\mathbb{R}^n \times \mathbf{U}_{[0:q-1]}(x^0)}$ we have
\begin{align}\label{eqn:cond_V_alpha}
    V(x^0, \boldu_{[0:q-1]}) - \alpha(x^0,\boldu_{[0:q-1]}) \geq 0;
\end{align}
\item for any $(x^0,\boldu_{[0:q-1]}) \in \mathbb{R}^n \times \mathbf{U}_{[0:q-1]}(x^0)$ there exists $\boldsymbol{\nu}_{[0:q+m-1]}\in \mathbb{R}^{p,q+m}$
with $\boldnu_{[l:q+l-1]} \in \mathbf{U}_{[0:q-1]}(x^l)$ for every $l \in \{0,1,\dots,m\}$, which steers $x^0$ to some $x^{m}$, such that
\begin{align}\label{eqn:dclf_orderm}
    &\frac{\sigma_m V\left(x^{m},\boldsymbol{\nu}_{[m:q+m-1]}\right) + \dots +\sigma_1 V\left(x^1,\boldsymbol{\nu}_{[1:q]}\right)}{m} -V(x^0, \boldu_{[0:q-1]}) \leq -\alpha(x^0,\boldu_{[0:q-1]}),
\end{align}
where $\sigma_m, \ldots, \sigma_1 \in \real_{\geq 0}$ satisfy~\eqref{eqn:sigma_ave}.
\end{enumerate}
\end{definition}
In contrast to the standard case, the interpretation of Definition~\ref{def:ho-dclf} is that the sequence of Lyapunov function values decreases not necessarily at every step, but on average every $m$ steps (compare with condition~\eqref{eqn:dclf_orderm}). Condition~\eqref{eqn:dclf_orderm} will provide a mean for relaxing the assumptions on the terminal conditions that are frequently imposed in standard MPC schemes.

Definition~\ref{def:ho-dclf} allows for a \hodclf~$V$ to depend on a control sequence $\boldu_{[0:q-1]}$. The variable $q$ determines how many components of the control sequence are considered. If $q=0$, the Lyapunov function only explicitly depends on the state. If $q=1$, the running cost of the optimal control problem could be taken as the Lyapunov function, e.g. $V(x, \boldu_{[0:0]}) = V(x, u) = \|x\|^2 + \|u\|^2$. If $q=N_p$, then the objective function of the optimal control problem could be taken as the Lyapunov function.

In the context of \emph{control} Lyapunov functions, the so-called small control property~\cite{ES:89} is often used. A generalization suitable for our framework is given below.

\begin{assumption}[Small Control Property]\label{assump:small_control_prop}
Con\-sid\-er system~\eqref{eqn:control_system} and let $V: \mathbb{R}^n \times \mathbb{R}^{p, q} \to \mathbb{R} $ be a \hodclf~of order $m$. We say that the small control property holds if for all $\varepsilon>0$, there exists $\delta>0$ such that the following conditions hold
\begin{enumerate}[i)]
    \item for all $0<\|x^{0}\|<\delta$, there exist $\boldu_{[0:q-1]} \in \mathbf{U}_{[0:q-1]}(x^{0})$ and $\boldnu_{[0:q+m-1]}$ with $\boldnu_{[l:q+l-1]} \in \mathbf{U}_{[0:q-1]}(x^{l})$ for every $l \in \{0,1,\dots,m\}$; 
    \item the controls $\boldu_{[0:q-1]}$ and $\boldnu_{[0:q+m-1]}$ satisfy $\max\{\|\boldu_{[0:q-1]}\|,\|\boldnu_{[0:q+m-1]}\|\}<\varepsilon$;
    \item the inequalities~\eqref{eqn:cond_V_alpha} and~\eqref{eqn:dclf_orderm} hold with $\boldu_{[0:q-1]}$ and $\boldnu_{[0:q+m-1]}$.
\end{enumerate}
\end{assumption}

The following results demonstrate how a Lyapunov function in the sense of Definition~\ref{def:ho-dclf} leads to stability.
\begin{theorem}\label{thm:ho-dclf-stab}
If there exists a \hodclf~of order $m$ for the control system~\eqref{eqn:control_system} and Assumption~\ref{assump:small_control_prop} holds, then for any $k_0\in \mathbb{N}$ and any state $\inistatetimev \in X$, there exists a feasible control sequence such that $\lim_{k \to \infty}x^k=0$.
\end{theorem}

\begin{theorem}\label{thm:ho-dclf_lyapstab}
Let $V: \mathbb{R}^n \times \mathbb{R}^{p, q} \to \mathbb{R}$ be a \hodclf~of order $m$ for the control system~\eqref{eqn:control_system} with positive weights $\sigma_i$ and suppose that Assumption~\ref{assump:small_control_prop} holds. Then there exists a control strategy which renders the origin asymptotically stable.  
\end{theorem}
The proof of both results can be found in Section~\ref{sec:stability_proofs}.
It is worth connecting the above results to the existing literature. In particular, a corollary of Theorem~\ref{thm:ho-dclf-stab} is obtained regarding so-called finite-step control Lyapunov functions, recently introduced in~\cite{NN-RG-LG-FW:20}. By definition, finite-step control Lyapunov functions decrease after finitely many, e.g. $m$, steps. This means they satisfy~\eqref{eqn:adc_q0} with $\sigma_{m-1},\dots,\sigma_1=0$.
\begin{corollary}
If there exists a finite-step control Lyapunov function for the control system \eqref{eqn:control_system} and Assumption~\ref{assump:small_control_prop} holds, then for any initial state $x^0 \in X$, there exists a feasible control sequence such that $\lim_{k\to \infty}x^k=0$.
\end{corollary}
The proof of this corollary follows immediately, since the definition of \hodclf s~includes finite-step control Lyapunov functions. With these results at hand, we are ready to introduce our novel MPC scheme.
\subsection{Flexible-step MPC scheme}
The concept of \hodclf s is now utilized to present our MPC scheme. Recall that Definition~\ref{def:ho-dclf} only guarantees the existence of a $(q+m)$-long sequence $\boldnu_{[0:q+m-1]}$ satisfying~\eqref{eqn:dclf_orderm}. These
controls need to be found through an optimization problem. More precisely, the flexible-step MPC scheme is based on the
optimal control problem defined in Problem~\ref{prob:ho-dclf} below where the condition~\eqref{eqn:dclf_orderm} has been added as a constraint.
Since the fraction in \eqref{eqn:dclf_orderm} represents a weighted average, we will refer to this constraint as the \emph{average decrease constraint} (\adc). 
The \adc~constraint guarantees, that there exists an index of descent $\ld$, $1 \leq \ld \leq m$, at which a descent of the \hodclf~is achieved, that is $V(x^{*\ld}, \boldu^*_{[\ld:q+\ld-1]}) -V(x,\boldu^{*-}_{[0:q-1]})\leq -\alpha(x,\boldu^{*-}_{[0:q-1]})$. 
Here, $\boldu^{*-}_{[0:q-1]}$ denotes the $q$-long control strategy of the previous iteration or more precisely, $\boldu^{*-}_{[0:q-1]} := \boldu^*_{[\ld:q+\ld-1]}$, where $\boldu^*_{[\ld:q+\ld-1]}$ was the optimal control strategy with the index of descent $\ld$ of the \hodclf~in that previous iteration. 
\begin{equation}\label{eq:display}
\hspace{-0.36cm}
    \begin{aligned}
    &\min \sum_{j=0}^{N_p -1} f_0(x^j,u^{j}) + \phi(x^{N_p})\\
    \ &\text{s.t. }  x^{j+1} = f(x^j,u^j), \ x^0 = x, \ \ u^j \in U,\,x^j \in X, \, x^{N_p} \in X^{N_p}, \\
    &\hphantom{\text{s.t. }}[u^l, \dots, u^{l+q-1}] \in \mathbf{U}_{[0:q-1]}(x^l) \text{ for } l=0,1,\dots,m\\
    &\hphantom{\text{s.t. }}\frac{1}{m}\Big(\sigma_{m}V(x^{m}, [u^m, \dots,u^{m-1+q}])
    + \sigma_{m-1}V(x^{m-1}, [u^{m-1}, \dots,u^{m-2+q}])
    + \ldots 
    +\sigma_1V(x^{1}, [u^1, \dots,u^{q}])\hspace{-0.05cm}\Big)\\
    &\hphantom{\text{s.t. }\frac{1}{m}\Big(} - V(x^{0},\boldu^{*-}_{[0:q-1]}) \leq -\alpha(x^0,\boldu^{*-}_{[0:q-1]})
    \end{aligned}
\end{equation}
\begin{problem}\label{prob:ho-dclf}
Choose the following parameters a-priori: $N_p \in \mathbb{N}, q \leq N_p$ with $q \in \mathbb{N}, m \leq N_p$ with $m \in \mathbb{N}_{>0}, X^{N_p} \subseteq X \subseteq \mathbb{R}^n$ with $0 \in \interior(X^{N_p}),\sigma_m, \ldots, \sigma_1 \in \mathbb{R}_{\geq 0}$ satisfying~\eqref{eqn:sigma_ave}, a positive semi-definite function $\phi: \mathbb{R}^n \to \mathbb{R}$ and a \hodclf~$V: \real^n \times \real^{p,q} \to \real$ of order $m$. Solve the finite-horizon optimal control problem~\eqref{eq:display}, where  $f_0: \mathbb{R}^n\times\mathbb{R}^p\to \mathbb{R}$ is positive definite and $0 \in \interior(U)$, $U \subseteq \mathbb{R}^p$. Here $x=x(k)$ is the current state and $\boldu^{*-}_{[0:q-1]}$ is the previous control strategy. The optimization is done over the $N = \max\{q+m, N_p\}$ control inputs $[u^0,u^{1}, \dots, u^{N-1}]$.
\end{problem}
Similar to prior literature and for the sake of simplicity, we will assume from now on that a bounded solution to Problem~\ref{prob:ho-dclf} exists, this should not be confused with recursive feasibility for which we provide results later on. 

Problem~\ref{prob:ho-dclf} is written for the case $q\neq0$; like in Definition~\ref{def:ho-dclf}, this can be adjusted for the case $q=0$. To be precise, in equation~\eqref{eq:display} the constraint involving the feasible set of controls needs to read $[u^0, \dots, u^{m-1}] \in \mathbf{U}_{[0:m-1]}(x^0)$. This constraint is already covered by $x^j \in X$ and $u^j \in U$ and hence does not have to be included in the case of $q=0$. The \adc~constraint becomes $\frac{1}{m}(\sum_{j=1}^mV(x^j)) - V(x^0) \leq - \alpha(x^0)$. From now on, we leave it to the reader to make the necessary adjustments for the case $q = 0$.

The flexible-step MPC scheme is given in Algorithm~\ref{algo:ho-dclf}. Note that Problem~\ref{prob:ho-dclf} needs to be solved as part of this algorithm. Like in any MPC scheme, any suitable solver for this nonlinear constrained optimization problem can be used.

\begin{algorithm}[ht]
\caption{Flexible-step MPC scheme}
\label{algo:ho-dclf}
\begin{algorithmic}[1]
\STATE set $k=\initimetimev$, measure the initial state $x(k_0)$ and choose an arbitrary $\boldu^{*-}_{[0:q-1]} \in \mathbf{U}_{[0:q-1]}(x(k_0))$
\STATE measure the current state $x(k)$ of~\eqref{eqn:control_system}
\STATE solve Problem~\ref{prob:ho-dclf} with $x=x(k)$ and obtain the optimal input ${\mathbf{u}^{*}_{[0:N-1]}}$, where $N=\max\{q+m,N_p\}$
\STATE choose an index $1\leq\ld\leq m$ for which $V(x^{*\ld},\mathbf{u}^*_{[\ld:q+\ld-1]}) - V(x,\boldu^{*-}_{[0:q-1]}) \leq -
    \alpha(x,\boldu^{*-}_{[0:q-1]})$ 
\STATE implement $\boldu^*_{[0:\ld-1]}=:[c^{k}_{mpc},\dots, c^{k+\ell_{decr}-1}_{mpc}]$ and redefine $\boldu^{*-}_{[0:q-1]}:=\boldu^{*}_{[\ld:q+\ld-1]}$ 
\STATE increase $k:=k+\ld$ and go to 2
\end{algorithmic}
\end{algorithm}

We find it useful to explain the idea behind our implementation with an example, illustrated in Figure~\ref{fig:expl_new_mpcimpl}. 
\begin{figure*}
\begin{subfigure}[b]{0.242\textwidth}
\begin{tikzpicture}[scale=0.45][x=0.5cm, y=0.5cm,domain=0:11,smooth]
   \draw [color=gray!30]  [step=5mm] (-0.3,-0.5) grid (8.5,12);
   \draw[->,thick] (0,0) -- (8,0) node[right] {$k$};
   \draw[->,thick] (0,0) -- (0,11) node[above] {$V$};
   \foreach \c in {1,2,...,7}{
     \draw (\c,-.1) -- (\c,.1) ;
   }
    \filldraw[LightGreen] (0,9) circle (8pt);
    \filldraw[DarkBlue] (0,9) circle (3pt) node[anchor=west] {$V(x^0)$}; 
    \draw[DarkBlue] (0,9) -- (1,8);
    \filldraw[DarkBlue] (1,8) circle (3pt) node[anchor=west] {$V(x^1)$};
    \draw[DarkBlue] (1,8) -- (2,6);
    \filldraw[DarkBlue] (2,6) circle (3pt) node[anchor=west] {$V(x^2)$};
    \draw[DarkBlue,dashed] (2,6) -- (3,7.5);
    \filldraw[DarkBlue] (3,7.5) circle (3pt) node[anchor=west] {$V(x^3)$};
    \draw[DarkBlue,dashed] (3,7.5) -- (4,10);
    \filldraw[DarkBlue] (4,10) circle (3pt) node[anchor=west] {$V(x^4)$};
\end{tikzpicture}
\caption{Optimization at ${k=0}$ \hphantom{platzhalter}}
\label{fig:expl_new_mpcimpl1}
\end{subfigure}
\begin{subfigure}[b]{0.242\textwidth}
\begin{tikzpicture}[scale=0.45][x=0.5cm, y=0.5cm,domain=0:11,smooth]
   \draw [color=gray!30]  [step=5mm] (-0.3,-0.5) grid (8.5,12);
   \draw[->,thick] (0,0) -- (8,0) node[right] {$k$};
   \draw[->,thick] (0,0) -- (0,11) node[above] {$V$};
   \foreach \c in {1,2,...,7}{
     \draw (\c,-.1) -- (\c,.1) ;
   }
    \filldraw[gray!40] (0,9) circle (8pt);
    \filldraw[black!70] (0,9) circle (3pt);
    \draw[black!50] (0,9) -- (1,8);
    \filldraw[black!70] (1,8) circle (3pt);
    \draw[black!50] (1,8) -- (2,6);
    \filldraw[LightGreen] (2,6) circle (8pt);
    \filldraw[LightBlue] (2,6) circle (3pt) node[anchor=south] {$V(x^0)$};
    \draw[black!50,dashed] (2,6) -- (3,7.5);
    \filldraw[black!70] (3,7.5) circle (3pt);
    \draw[black!50,dashed] (3,7.5) -- (4,10);
    \filldraw[black!70] (4,10) circle (3pt);
    \filldraw[LightBlue] (3,3) circle (3pt) node[anchor=west] {$V(x^1)$};
    \draw[LightBlue] (2,6) -- (3,3);
    \filldraw[LightBlue] (4,6.4) circle (3pt) node[anchor=west] {$V(x^2)$};
    \draw[LightBlue,dashed] (3,3) -- (4,6.4);
    \filldraw[LightBlue] (5,8.5) circle (3pt) node[anchor=west] {$V(x^3)$};
    \draw[LightBlue,dashed] (4,6.4) -- (5,8.5);
    \filldraw[LightBlue] (6,10.5) circle (3pt) node[anchor=west] {$V(x^4)$};
    \draw[LightBlue,dashed] (5,8.5) -- (6,10.5);
\end{tikzpicture}
\caption{Optimization at ${k=2}$ \hphantom{platzhalter}}
\label{fig:expl_new_mpcimpl2}
\end{subfigure}
\begin{subfigure}[b]{0.242\textwidth}
\begin{tikzpicture}[scale=0.45][x=0.5cm, y=0.5cm,domain=0:11,smooth]
   \draw [color=gray!30]  [step=5mm] (-0.3,-0.5) grid (8.5,12);
   \draw[->,thick] (0,0) -- (8,0) node[right] {$k$};
   \draw[->,thick] (0,0) -- (0,11) node[above] {$V$};
   \foreach \c in {1,2,...,7}{
     \draw (\c,-.1) -- (\c,.1) ;
   }
    \filldraw[gray!40] (0,9) circle (8pt);
    \filldraw[black!70] (0,9) circle (3pt);
    \draw[black!50] (0,9) -- (1,8);
    \filldraw[black!70] (1,8) circle (3pt);
    \draw[black!50] (1,8) -- (2,6);
    \filldraw[gray!40] (2,6) circle (8pt);
    \filldraw[black!70] (2,6) circle (3pt);
    \draw[black!50, dashed] (2,6) -- (3,7.5);
    \filldraw[black!70] (3,7.5) circle (3pt);
    \draw[black!50, dashed] (3,7.5) -- (4,10);
    \filldraw[black!70] (4,10) circle (3pt);
    \filldraw[LightGreen] (3,3) circle (8pt);
    \filldraw[DarkGreen] (3,3) circle (3pt) node[anchor=west] {$V(x^0)$};
    \draw[black!50] (2,6) -- (3,3);
    \filldraw[black!70] (4,6.4) circle (3pt);
    \draw[black!50,dashed] (3,3) -- (4,6.4);
    \filldraw[black!70] (5,8.5) circle (3pt);
    \draw[black!50,dashed] (4,6.4) -- (5,8.5);
    \filldraw[black!70] (6,10.5) circle (3pt);
    \draw[black!50,dashed] (5,8.5) -- (6,10.5);
    \draw[DarkGreen] (3,3) -- (4,4);
    \filldraw[DarkGreen] (4,4) circle (3pt) node[anchor=west] {$V(x^1)$};
    \filldraw[DarkGreen] (5,5) circle (3pt) node[anchor=east] {$V(x^2)$};
    \draw[DarkGreen] (4,4) -- (5,5);
    \draw[DarkGreen] (5,5) -- (6,6);
    \filldraw[DarkGreen] (6,6) circle (3pt) node[anchor=west] {$V(x^3)$};
    \draw[DarkGreen] (6,6) -- (7,1);
    \filldraw[DarkGreen] (7,1) circle (3pt) node[anchor=east] {$V(x^4)$};
\end{tikzpicture}
\caption{Optimization at ${k=3}$ \hphantom{platzhalter}}
\label{fig:expl_new_mpcimpl3}
\end{subfigure}
\begin{subfigure}[b]{0.242\textwidth}
\begin{tikzpicture}[scale=0.45][x=0.5cm, y=0.5cm,domain=0:11,smooth]
   \draw [color=gray!30]  [step=5mm] (-0.3,-0.5) grid (8.5,12);
   \draw[->,thick] (0,0) -- (8,0) node[right] {$k$};
   \draw[->,thick] (0,0) -- (0,11) node[above] {$V$};
   \foreach \c in {1,2,...,7}{
     \draw (\c,-.1) -- (\c,.1) ;
   }
    \draw[yellow!70,line width=1.1mm] (0,9) -- (1,8);
    \draw[yellow!70,line width=1.1mm] (1,8) -- (2,6);
    \draw[yellow!70,line width=1.1mm] (2,6) -- (3,3);
    \draw[yellow!70,line width=1.1mm] (3,3) -- (4,4);
    \draw[yellow!70,line width=1.1mm] (4,4) -- (5,5);
    \draw[yellow!70,line width=1.1mm] (5,5) -- (6,6);
    \draw[yellow!70,line width=1.1mm] (6,6) -- (7,1);
    \filldraw[LightGreen] (0,9) circle (8pt);
    \filldraw[LightGreen] (2,6) circle (8pt);
    \filldraw[LightGreen] (3,3) circle (8pt);
    \filldraw[LightGreen] (7,1) circle (8pt);
    \filldraw[black] (0,9) circle (3pt) node[anchor=west] {$V(x(0))$};
    \draw[black] (0,9) -- (1,8);
    \filldraw[black] (1,8) circle (3pt) node[anchor=west] {$V(x(1))$};
    \draw[black] (1,8) -- (2,6);
    \filldraw[black] (2,6) circle (3pt) node[anchor=west] {$V(x(2))$};
    \filldraw[black] (3,3) circle (3pt) node[anchor=west] {$V(x(3))$};
    \draw[black] (2,6) -- (3,3);
    \draw[black] (3,3) -- (4,4);
    \filldraw[black] (4,4) circle (3pt) node[anchor=west] {$V(x(4))$};
    \filldraw[black] (5,5) circle (3pt) node[anchor=east] {$V(x(5))$};
    \draw[black] (4,4) -- (5,5);
    \draw[black] (5,5) -- (6,6);
    \filldraw[black] (6,6) circle (3pt) node[anchor=west] {$V(x(6))$};
    \draw[black] (6,6) -- (7,1);
    \filldraw[black] (7,1) circle (3pt) node[anchor=east] {$V(x(7))$};
\end{tikzpicture}
\caption{Lyapunov function values along actual states}
\label{fig:final_trajectory_expl_mpc_newimpl}
\end{subfigure}
\caption{Illustration of proposed MPC scheme: Consider Problem~\ref{prob:ho-dclf} with $m=4$ and $q=0$. The initial state is ${\color{DarkBlue}x^0}=x(0)$, whose Lyapunov function value is depicted in (A) and highlighted in green. After solving the finite-horizon optimal control problem, we obtain four predicted states and their corresponding Lyapunov function values ({\color{DarkBlue}$V(x^0),V(x^1),V(x^2),V(x^3),V(x^4)$}). Since there are multiple time indices for which the Lyapunov function decreases, we choose $\ld$ here as the index where the greatest descent occurs. We implement $\ld=2$ components of the control sequence and, consequently, declare {\color{DarkBlue}$x^2$} as the new initial state for the finite-horizon optimal control problem at time $k=2$. 
This problem is solved in (B) and we obtain again four states and their corresponding Lyapunov function values ({\color{LightBlue}$V(x^0),V(x^1),V(x^2),V(x^3),V(x^4)$}). We repeat the scheme and after solving Problem~\ref{prob:ho-dclf} three times, we obtain the trajectory of the closed-loop states and their Lyapunov function values shown in (D).}
\label{fig:expl_new_mpcimpl}
\end{figure*}
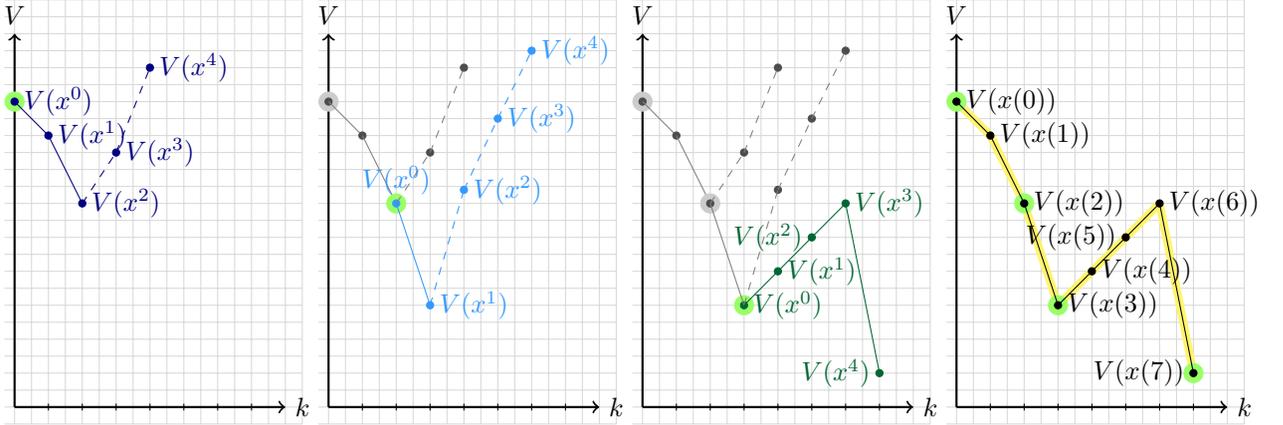
Consider Problem~\ref{prob:ho-dclf} with $m=4$ and a \hodclf~$V$ with $q=0$. Let $x^0=x(0)$ be the initial state. After solving Problem~\ref{prob:ho-dclf}, we obtain a trajectory of states ($x^0,x^1,x^2,x^3,x^4$). The corresponding Lyapunov function values are shown in Figure \ref{fig:expl_new_mpcimpl1} in dark blue. By definition of \adc, we know that at least one of the four future Lyapunov function values will be strictly less than $V(x^0)$ (see Proposition~\ref{thm:ex_subsqn_stability}). In order to ultimately reach convergence of the state to zero, we need to implement the optimal control sequence until a descent of the Lyapunov function value occurs. In this example, we choose to implement until the greatest descent occurs. In Figure~\ref{fig:expl_new_mpcimpl1}, this is achieved at state $x^2$, thus, we implement the optimal control sequence for two time steps and declare the actual states $x(1)=x^1, x(2)=x^2$. It is worthwhile to mention that in Figure~\ref{fig:expl_new_mpcimpl1}, the predicted Lyapunov function value after the fixed amount of four steps is strictly greater than the Lyapunov function value at the beginning of the prediction, i.e. $V(x^4) > V(x^0)$. This would not be feasible in an MPC scheme based on a finite-step Lyapunov decay condition. Now the new initial state for the finite-horizon optimal control problem becomes $x(2)$, highlighted in light green. We repeat this process, as demonstrated in Figure~\ref{fig:expl_new_mpcimpl2}. The greatest descent is achieved after one time step. Thus, we only implement the optimal control sequence for one step, $x(3)=x^1$. We define the new initial state and proceed as before. This time, the greatest descent is achieved after four time steps, see Figure~\ref{fig:expl_new_mpcimpl3}. The yellow trajectory in Figure \ref{fig:final_trajectory_expl_mpc_newimpl} shows the Lyapunov function values along the actual states of our example after three iterations of Algorithm~\ref{algo:ho-dclf}.

At this point, we would like to clarify steps four and five of Algorithm \ref{algo:ho-dclf}. By Proposition~\ref{thm:ex_subsqn_stability}, \adc~guarantees a descent of the Lyapunov function value, step four captures where it occurs through the index $\ld$ and in step five the control sequence is implemented until the descent.  
The term \textit{flexible-step} MPC is chosen to emphasize that steps four and five lead to implementation of possibly \textit{different} number of steps in each iteration. Hence, we have arrived at a flexible-step implementation without adding an integer decision variable to Problem~\ref{prob:ho-dclf}. Adding instead the constraint \adc, has enabled us to easily determine the index of descent after the optimization, as it is stated in Algorithm 1.

\begin{remark}
{\em
Algorithm~\ref{algo:ho-dclf} generates a feedback law $c_{mpc}^k(x(\pi(k)),\boldu^{*-}(\pi(k)))$ by concatenating a \emph{flexible} amount of components of the optimal solution $\boldu^*$ to Problem~\ref{prob:ho-dclf}. This feedback law is time and state dependent, common in continuous-time sampled-data MPC~\cite{FF-LM-EG:07,FF:03}. Here, $\pi(k)$ is the time step of the most recent state measurement, i.e. the beginning of each optimization instance of Problem~\ref{prob:ho-dclf} (step three in Algorithm~\ref{algo:ho-dclf}) and $\boldu^{*-}(\pi(k))$ is part of the control strategy of the previous iteration (assigned in step five of the algorithm). The novel implementation in Algorithm~\ref{algo:ho-dclf} naturally leads to a time-varying feedback law, and hence system, because the actual states generated by the flexible-step MPC scheme evolve according to the dynamics 
\begin{align}\label{eqn:flexstepdynamics}
\begin{aligned}
    x(k+1) &= f\big(x(k),c_{mpc}^k(x(\pi(k)),\boldu^{*-}(\pi(k)))\big),
\end{aligned}
\end{align}
for $k \geq k_0$. 
Alternatively, we can extend system~\eqref{eqn:flexstepdynamics} by an augmented variable $y(k) \in \mathbb{R}^n\times \mathbb{R}^{p,q}$, which captures the state and $\boldu^{*-}$ at the optimization instances, and allows us to write the closed-loop dynamics as follows 
\begin{align}\label{eqn:relyx}
    &x(k+1) = f(x(k),c^k_{mpc}\big(y(k))\big)\nonumber\\
    &y(k+1) = \left\{\begin{array}{ll} \Big(f\big(x(k),c^k_{mpc}(y(k))\big),\boldu^{*-}(k+1)\Big), &\text{if optimization instance occurs at } k+1\\
    y(k), & \text{otherwise}\end{array}\right.
\end{align}
with $y(k_0)=(x(k_0),\boldu^{*-}(k_0))$. In the previous example, illustrated in Figure~\ref{fig:expl_new_mpcimpl}, the optimization instances occur at $k=0,2,3,7$, thus,
\begin{align*}
    [y(0),y(1), \dots ,y(7)]= &[(x(0),\boldu^{*-}(0)), (x(0),\boldu^{*-}(0)), (x(2),\boldu^{*-}(2)),(x(3),\boldu^{*-}(3)), (x(3),\boldu^{*-}(3)), \\
    &\ (x(3),\boldu^{*-}(3)), (x(3),\boldu^{*-}(3)), (x(7),\boldu^{*-}(7))].
\end{align*}
The first component of $y(k)$ is determined by $x(k)$ through~\eqref{eqn:relyx} and the second component captures $\boldu^{*-}$. Since both components will converge to zero if $x(k)$ converges to zero (the zero control is a feasible control with minimal cost for Problem~\ref{prob:ho-dclf}) and since the state extension leads to a cumbersome formulation, we consider only the state variable $x(k)$ in the following. 
\oprocend
}
\end{remark}

\begin{remark}\label{rkm:changeustarminus}
\em{
The way we assign $\boldu^{*-}(\pi(k))$ in Algorithm~\ref{algo:ho-dclf} allows us to guarantee a chain of inequalities, similar to~\eqref{eqn:Vstrictineqs}, needed for the stability proof. It is possible to change this assignment while still guaranteeing the chain of inequalities. In fact, if we change step five in Algorithm~\ref{algo:ho-dclf} to: implement $\boldu^*_{[0:\ld-1]}=:[c^{k}_{mpc},\dots, c^{k+\ell_{decr}-1}_{mpc}]$, find $\bar \boldu_{[0:q-1]} \in \mathbf{U}_{[0:q-1]}(x^{*\ld})$ such that 
\begin{align*}
    V(x^{*\ld},\bar \boldu_{[0:q-1]}) - V(x,\boldu^{*-}_{[0:q-1]}) \leq -
    \alpha(x,\boldu^{*-}_{[0:q-1]})
\end{align*}
and redefine $\boldu^{*-}_{[0:q-1]}:=\bar \boldu_{[0:q-1]}$, then the argument of the stability proof still applies.  
\oprocend
}
\end{remark}
We next state a sequence of results related to Algorithm~\ref{algo:ho-dclf}. We start with the following assumption. It is a generalization of the usual assumption made in standard MPC, i.e. the existence of a feasible controller which renders the set $X^{N_p}$ invariant and achieves a descent~\cite{DM-JR-CR-PS:00, JR-DM-MD:17}. 
\begin{assumption}\label{assump:loc_feedback_recfeas}
Consider the sets $U\subseteq \mathbb{R}^p, X^{N_p}\subseteq X \subseteq \mathbb{R}^n$ with $0 \in  \interior(U)$ and $0 \in \interior(X^{N_p})$. 
Let $V: \mathbb{R}^n \times \mathbb{R}^{p,N_p} \to \mathbb{R}$ be a \hodclf~of order $m$ for system~\eqref{eqn:control_system}. Assume that for any $(x^0, \boldu_{[0:N_p-1]}) \in X \times \mathbf{U}_{[0:N_p-1]}(x^0)$ there exists a feedback $\mathbf{c}: X^{N_p} \to \mathbb{R}^{p,m}, \tilde x \mapsto \mathbf{c}(\tilde x)$ such that for all $\tilde x \in X^{N_p}$
\begin{enumerate}
    \item the $m$ components of $\mathbf{c}(\tilde x)$ satisfy $c_0(\tilde x),c_1(\tilde x),\dots,c_{m-1}(\tilde x) \in U$;
    \item with $x^0=\tilde x$, we have ${x^{j+1} = f(x^j, c_j(\tilde x))}$ $\in X^{N_p}$ for all $j=0,\dots,m-1$;
    \item $[\boldu_{[0:N_p-1]},\mathbf{c}(\tilde x)]$ satisfies \adc~\eqref{eqn:dclf_orderm}, i.e. the control sequence $[\boldu_{[0:N_p-1]},\mathbf{c}(\tilde x)]$ steers $x^0$ to some $x^m$ such that
    \begin{align*}
    &\frac{1}{m}\left(\sigma_m V\left(x^{m},[\boldu_{[m:N_p-1]},\mathbf{c}(\tilde x)]\right)+ \dots +\sigma_1 V\left(x^1,[\boldu_{[1:N_p-1]},c_0(\tilde x)]\right)\right)-V(x^0, \boldu_{[0:N_p-1]}) \\
    &\hspace{2cm}\leq -\alpha(x^0, \boldu_{[0:N_p-1]}).
    \end{align*} 
\end{enumerate}
\end{assumption}
The following theorem guarantees recursive feasibility as defined in e.g.~\cite{DM:14}. 
\begin{theorem}\label{thm:rec_feas}
Let $V: \mathbb{R}^n\times \mathbb{R}^{p, N_p} \to \mathbb{R}$ be a \hodclf~of order $m$ for system~\eqref{eqn:control_system}. Suppose that Assumption~\ref{assump:loc_feedback_recfeas} is satisfied and that Problem~\ref{prob:ho-dclf} with $x=x(\initimetimev)$ is feasible. Then the MPC
scheme given in Algorithm~\ref{algo:ho-dclf} is recursively feasible.
\end{theorem}
\begin{proof}
To prove recursive feasibility, we make use of the terminal set $X^{N_p}$. As a first step, we add the terminal constraint of Problem~\ref{prob:ho-dclf} to the sets $\mathbf{U}_{[0:N_p-1]}(x^l)$ with $x^l \in X$ for $l = 0,1,\dots,m $
\begin{align*}
    &\mathbf{U}_{[0:N_p-1]}(x^l)= \{\boldu_{[0:N_p-1]}: \text{ with initial state } x^l, u_j \in U, x^{j+1} = f(x^j,u_j) \in X,\\
    &\hspace{4.5cm}j = 0, \dots, N_p-1, x^{N_p} \in X^{N_p}\}.    
\end{align*} 
Consider now Algorithm~\ref{algo:ho-dclf} starting at some non-negative integer $ k$, and suppose that Problem~\ref{prob:ho-dclf} was feasible in the previous iteration. 
We show that Problem~\ref{prob:ho-dclf} stays feasible after executing Algorithm~\ref{algo:ho-dclf}.
To that end, suppose we arrived at the state $x(k)\in X$ and that the previous control strategy is given by $\boldu^{*-}_{[0:N_p-1]}$. Thus, we define the initial state of Problem~\ref{prob:ho-dclf} as $x^0=x(k)$. Since $\boldu^{*-}_{[0:N_p-1]}\in \mathbf{U}_{[0:N_p-1]}(x(k))$, the states steered by $\boldu^{*-}_{[0:N_p-1]}$ satisfy 
\[
x^{1}, \dots, x^{N_p-1} \in X, \quad 
x^{N_p} \in X^{N_p}.
\]
Additionally, all components of $\boldu^{*-}_{[0:N_p-1]}$ are elements of $U$. By considering $\tilde x=x^{N_p} \in X^{N_p}$ in Assumption~\ref{assump:loc_feedback_recfeas}, the control sequence $[\boldu^{*-}_{[0:N_p-1]},\mathbf{c}(x^{N_p})]$ and the corresponding states satisfy \adc. Moreover, the controller renders $X^{N_p}$ invariant. Therefore, we have  $x^{N_p+1}, \dots, x^{N_p+m}\in X^{N_p}\subseteq X$. Hence, the choice of $[\boldu^{*-}_{[0:N_p-1]},\mathbf{c}(x^{N_p})]$ yields a feasible controller. 
\end{proof}
\begin{theorem}\label{thm:algo_convtozero}
Assume that Problem~\ref{prob:ho-dclf} is initially feasible with $x=x(k_0)$, $k_0 \in \mathbb{N}$ and Algorithm~\ref{algo:ho-dclf} generates a control sequence which satisfies the conditions in Assumption~\ref{assump:small_control_prop}. Then under Algorithm~\ref{algo:ho-dclf}, 
we have that $x(k) $ converges to zero as $ k $ goes to infinity.
\end{theorem}
\begin{theorem}\label{thm:algo_conv}
Assume that Problem~\ref{prob:ho-dclf} is initially feasible with $x=x(k_0)$, $k_0 \in \mathbb{N}$ and Algorithm~\ref{algo:ho-dclf} generates a control sequence which satisfies the conditions in Assumption~\ref{assump:small_control_prop}. Then under Algorithm~\ref{algo:ho-dclf}, 
we have that $x(k) $ converges to zero as $ k $ goes to infinity. Moreover, if all the weights $\sigma_i$ in \adc~are positive, then the origin is rendered asymptotically stable by Algorithm~\ref{algo:ho-dclf}.
\end{theorem}

The proofs of Theorem~\ref{thm:algo_convtozero} and~\ref{thm:algo_conv} are a consequence of Theorem~\ref{thm:ho-dclf-stab} and~\ref{thm:ho-dclf_lyapstab} and are presented in Appendix~\ref{sec:appendix} for the sake of completeness. Now that we have established the properties of the flexible-step MPC scheme, some remarks are in order.

\begin{remark}\label{rmk:recov_stand_MPC}
{\em 
To recover the standard MPC scheme from our framework, we consider the case of $q=N_p, m=1$ and $\sigma_1=1$. An appropriate terminal cost $\phi(x^{N_p})$, i.e. an (ordinary) control Lyapunov function on $X^{N_p}$, ultimately makes the objective function a suitable choice for the \hodclf~$V$. Since $m=1$, \adc~reads
\begin{align}\label{eqn:adc_stdMPC}
\begin{aligned}
    &V(x^1,[u^1, \dots, u^{N_p}]) - V(x^0,\boldu^{*-}_{[0:N_p-1]})\leq -\alpha(x^0,\boldu^{*-}_{[0:N_p-1]}),
\end{aligned}
\end{align}
where $\boldu^{*-}_{[0:N_p-1]}$ is the previous control strategy. Theorem~\ref{thm:rec_feas} guarantees that there exists a feasible $u^{N_p}$, such that the resulting controller $[\boldu^{*-}_{[0:N_p-1]},u^{N_p}]$ will satisfy \adc. This implies that with the standard MPC assumption, which was generalized by Assumption~\ref{assump:loc_feedback_recfeas}, constraint~\eqref{eqn:adc_stdMPC} is always satisfied, meaning that it does not further restrict the feasible controls. Here, the constraint \adc~demands a strict descent of the objective function $V$ after \emph{one} time step; this corresponds exactly to the standard MPC setting.  
\oprocend
}
\end{remark}

\begin{remark}
{\em
In the standard MPC setting, as discussed in Remark~\ref{rmk:recov_stand_MPC}, the Lyapunov function is chosen to be the objective function. In Problem~\ref{prob:ho-dclf} the Lyapunov function can generally be a different function. Ultimately, both, the optimization and the stabilization, seek zero state and control, nevertheless, focussing solely on one may hurt the overall objective. It has already been shown in the past~\cite{JP-VN-JD:98,1AB:98,MS-MD:90,PM-NE-PC:05} that it can be beneficial to choose a different Lyapunov function than the objective function to guarantee stability. However, the consideration of \hodclf s is novel, as they do not only depend on the state but also on the control input. Moreover, note that if $m=N_p$ and $q$ is greater than zero, the last function value that is considered in \adc~in Problem~\ref{prob:ho-dclf} is $V(x^{N_p}, [u^{N_p}, \dots,u^{N_p+q-1}])$. In other words, we will consider control inputs which go beyond the prediction horizon and help us decrease $V$ on average. This occurs in general if $m + q>N_p$.\oprocend}
\end{remark}

\begin{remark}
\em 
As mentioned in the previous remark, the proposed flexible-step MPC does not rely on the objective function as the Lyapunov function. Instead, the proofs of Theorem~\ref{thm:algo_convtozero} and \ref{thm:algo_conv} utilize the properties of a \hodclf~and not the positive definiteness of the cost function $f_0$. Assuming positive definiteness allows us, however, to recover standard MPC from flexible-step MPC (Remark~\ref{rmk:recov_stand_MPC}). \oprocend
\end{remark}

\section{Stability proofs}\label{sec:stability_proofs}
In this section, we provide the proofs of the stability results of Section~\ref{sec:notions_of_stab}. To obtain a better intuition about Algorithm~\ref{algo:ho-dclf}, the reader can skip this section in the first read and instead visit Section~\ref{sec:application_nonholo_sys} for case studies. 

To prove the first result, we need the following proposition.
\begin{proposition}\label{thm:ex_subsqn_stability}
If there exists a \hodclf~of order $m$ for the control system~\eqref{eqn:control_system}, then for any $(x^0,\boldu_{[0:q-1]}) \in$  ${\mathbb{R}^n\times \mathbf{U}_{[0:q-1]}(x^0})$ there exists $\boldsymbol{\nu}_{[0:q+m-1]}$ such that 
\begin{align*}
    V(x^l,\boldsymbol{\nu}_{[l:q+l-1]}) - V(x^0,\boldu_{[0:q-1]}) \leq -
    \alpha(x^0,\boldu_{[0:q-1]})
\end{align*}
holds for some $l \in \{1,\dots, m\}$ with $\boldnu_{[l:q+l-1]}\in \mathbf{U}_{[0:q-1]}(x^l)$.
\end{proposition}
\begin{proof}
To ease notation, we present this proof for the case $m=2$; the general case follows similarly. Let $x^0 \in \mathbb{R}^n$ be some arbitrary state and let
$\boldu_{[0:q-1]}$ be an arbitrary $q$ long control sequence belonging
to the set $\mathbf{U}_{[0:q-1]}(x^0)$.We claim that condition~\eqref{eqn:dclf_orderm} and \eqref{eqn:sigma_ave} imply that there exists a control sequence $\boldsymbol{\nu}_{[0:q+1]}$, which steers $x^0$ to some $x^2$, such that either
\begin{align*}
    V(x^1,\boldsymbol{\nu}_{[1:q]}) - V(x^0,\boldu_{[0:q-1]}) \leq - \alpha(x^0,\boldu_{[0:q-1]}) 
\end{align*}
or 
\begin{align*}
 V(x^2,\boldsymbol{\nu}_{[2:q+1]})- V(x^0,\boldu_{[0:q-1]}) \leq - \alpha(x^0,\boldu_{[0:q-1]})    
\end{align*}
holds, meaning that either $V(x^1,\boldsymbol{\nu}_{[1:q]})$ or $V(x^2,\boldsymbol{\nu}_{[2:q+1]})$ is strictly less than $V(x^0,\boldu_{[0:q-1]})$, as $\alpha$ is positive definite. Suppose otherwise. Then
\begin{align*}
    &\frac{\sigma_2V(x^2,\boldsymbol{\nu}_{[2:q+1]}) + \sigma_1V(x^1,\boldsymbol{\nu}_{[1:q]})}{2} - V(x^0,\boldu_{[0:q-1]}) +\alpha(x^0,\boldu_{[0:q-1]})\\[5pt]
    & \hspace{10pt} >\frac{\sigma_2(V(x^0,\boldu_{[0:q-1]})- \alpha(x^0,\boldu_{[0:q-1]}))}{2} + \frac{\sigma_1(V(x^0,\boldu_{[0:q-1]})- \alpha(x^0,\boldu_{[0:q-1]}))}{2} \\
    & \hspace{21pt} - V(x^0,\boldu_{[0:q-1]})     + \alpha(x^0,\boldu_{[0:q-1]})\\[5pt]
    & \hspace{10pt}= \left(\frac{\sigma_2 + \sigma_1}{2}-1\right)(V(x^0,\boldu_{[0:q-1]})-\alpha(x^0,\boldu_{[0:q-1]}))\\
    & \hspace{10pt}\geq\hspace{0.1cm}0
\end{align*}
for all control sequences $\boldnu_{[0:q+1]}$, where the last inequality follows by the virtue of~\eqref{eqn:cond_V_alpha} and \eqref{eqn:sigma_ave}. The derived inequality contradicts \eqref{eqn:dclf_orderm}, finishing the proof. 
\end{proof}
Proposition~\ref{thm:ex_subsqn_stability} is important in what follows as it proves the  existence of an index $l \in \{1, \dots, m\}$ where we have a decay of the generalized discrete-time control Lyapunov function. Using Proposition \ref{thm:ex_subsqn_stability}, we now prove Theorem~\ref{thm:ho-dclf-stab}; the proof is a modification of a standard Lyapunov argument, see for instance \cite{HK:02}.
\begin{proof}[Proof of Theorem~\ref{thm:ho-dclf-stab}]
Fix $k_0 \in \mathbb{N}$ and $\inistatetimev \in X$. If the state $\inistatetimev$ is zero, there is nothing to show. So we may assume $x^{k_0}\neq 0$. Suppose, by contradiction, that there does not exist a feasible control sequence which steers $x^k$ to zero as $k$ goes to infinity. Note that this implies that $x^k$ can neither be equal to nor get arbitrarily close to zero for some $k \geq k_0$ as we elaborate at the end of this proof. 
Our goal is to define a subsequence 
\begin{align}\label{eqn:conv_subsqn}
    \{(x^{k_n}, \boldu^{k_n}_{[0:q-1]})\}_{n=0}^{\infty}    
\end{align}
with the help of Proposition~\ref{thm:ex_subsqn_stability}, so that the Lyapunov function values evaluated at these states and controls strictly decrease. To this end, let $\mathbf{u}_{[0:q-1]} \in \mathbf{U}_{[0:q-1]}(x^{k_0})$ be a fixed, but arbitrary, control sequence. By Proposition~\ref{thm:ex_subsqn_stability}, there exists $\boldsymbol{\nu}^{k_0}_{[0:q+m-1]}$, which steers $x^{k_0}$ to $x^{k_0+m}$ along~\eqref{eqn:control_system}, such that for some $l \in \{1,\dots, m\}$
\begin{align}\label{eqn:stab_proof_ldescent}
    V(x^{k_0+l},\boldsymbol{\nu}^{k_0}_{[l:q+l-1]}) < V(x^{k_0},\boldu_{[0:q-1]}),    
\end{align}
where $\boldnu^{k_0}_{[l:q+l-1]} \in \mathbf{U}_{[0:q-1]}(x^{k_0+l})$.
As a result, we can define the first two members of our subsequence:
\begin{align*}
    \boldsymbol{\omega}^{k_0}&:=(x^{k_0},\boldu_{[0:q-1]})\\
    \boldsymbol{\omega}^{k_1}&:= (x^{k_0+l^{k_0}},\boldsymbol{\nu}^{k_0}_{[l^{k_0}:q+l^{k_0}-1]}),
    \intertext{where}
    l^{k_0} &= \argmin_{l=
1,\dots,m}V(x^{k_0+l},\boldsymbol{\nu}^{k_0}_{[l:q+l-1]}) - V(\boldsymbol{\omega}^{k_0}).
\end{align*}
We define $\boldsymbol{\omega}^{k_1}$ in this fashion because there might be multiple pairs of states and controls that achieve a strict descent; in this proof we choose the pair with the greatest descent as $\boldsymbol{\omega}^{k_1}$. Other choices are possible, as long as equation~\eqref{eqn:stab_proof_ldescent} is satisfied. To define the rest of the subsequence, we proceed as follows. Suppose $\boldsymbol{\omega}^{k_n}$ has been defined. By Proposition~\ref{thm:ex_subsqn_stability}, there exists again $\boldsymbol{\nu}^{k_n}_{[0:q+m-1]}$ such that for some $l\in \{1,\dots,m\}$ the Lyapunov function value is strictly below $V(\boldsymbol{\omega}^{k_n})$ and $\boldnu^{k_n}_{[l:q+l-1]} \in \mathbf{U}_{[0:q-1]}(x^{k_n+l})$. We can define $\boldsymbol{\omega}^{k_{n+1}}$ accordingly
\begin{align*}
    \boldsymbol{\omega}^{k_{n+1}} &:= (x^{k_n+l^{k_n}},\boldnu^{k_n}_{[l^{k_n}:q+l^{k_n}-1]}),
    \intertext{where }
    l^{k_n} &= \argmin_{l= 1,\dots,m}V(x^{k_n+l},\boldnu^{k_n}_{[l:q+l-1]}) - V(\boldsymbol{\omega}^{k_n}).
\end{align*}
Along the subsequence $\{\boldsymbol{\omega}^{k_n}\}_{n=0}^{\infty}$, we obtain the following inequalities
\begin{align}\label{eqn:Vstrictineqs}
    V(\boldsymbol{\omega}^{k_0}) > V(\boldsymbol{\omega}^{k_1})&> \dots > V(\boldsymbol{\omega}^{k_n}) > \dots \geq 0.
\end{align}
The last inequality follows from the non-negativity of the \hodclf~$V$. Hence, the sequence $\{V(\boldsymbol{\omega}^{k_n})\}_{n=0}^\infty$ converges to some $ \delta \in \mathbb{R}$ as $n \to \infty$. Since $x^{k_n}\neq 0$ for all $n \in \mathbb{N}$ and $x^{k_n}$ also cannot get arbitrarily close to zero for some $n \in \mathbb{N}$, it follows that $\delta >0$ and that 
\begin{align}\label{eqn:xkn_in_C}
    \delta \leq V(\boldsymbol{\omega}^{k_n}) \leq V(\boldsymbol{\omega}^{k_0}) \hspace{0.3cm} \forall n \in \mathbb{N}.
\end{align} 
Next, we define 
\begin{align*}
    &C:=\{(z^0,\boldsymbol{\zeta}) \in \mathbb{R}^n\times \mathbf{U}_{[0:q-1]}(z^0): \delta \leq V(z^0,\boldsymbol{\zeta}) \leq V(\boldsymbol{\omega}^{k_0})\}.
\end{align*}
Since $V$ is continuous and by~\eqref{eqn:cond_V_alpha} radially unbounded, $C$ is closed and bounded, and hence a compact subset of $\mathbb{R}^n\times \mathbb{R}^{p,q}$. Since there exist states and controls such that  \eqref{eqn:xkn_in_C} follows, it holds $\boldsymbol{\omega}^{k_n} \in C$ for all $n \in \mathbb{N}$. For $(z^0,\boldsymbol{\zeta}_{[0:q-1]}) \in C$, Proposition~\ref{thm:ex_subsqn_stability} lets us conclude that there exists a control sequence $\boldsymbol{\xi}_{[0:q+m-1]}$ and some $l \in \{1,\dots, m\}$ for which $V(z^l,\boldsymbol{\xi}_{[l:q+l-1]})-V(z^0,\boldsymbol{\zeta}_{[0:q-1]}) \leq -\alpha(z^0,\boldsymbol{\zeta}_{[0:q-1]})$ and $\boldsymbol{\xi}_{[l:q+l-1]} \in \mathbf{U}_{[0:q-1]}(z^l)$.
Therefore,
\begin{align*}
    &\min_{l=1,\dots,m} \hspace{-0.1cm}V(z^l,\boldsymbol{\xi}_{[l:q+l-1]}) - V(z^0,\boldsymbol{\zeta}_{[0:q-1]}) \\
    & \hspace{40pt} \leq -\alpha(z^0,\boldsymbol{\zeta}_{[0:q-1]}), \hspace{0.3cm} \hspace{0.5cm}\forall (z^0,\boldsymbol{\zeta}_{[0:q-1]}) \in C\\
    \intertext{and consequently,}
    &\sup_{(z^0,\boldsymbol{\zeta}) \in C} \left[\min_{l=1,\dots,m}\hspace{-0.1cm} V(z^l,\boldsymbol{\xi}_{[l:q+l-1]}) - V(z^0,\boldsymbol{\zeta}_{[0:q-1]})\right] \\
    & \hspace{40pt} \leq \sup_{(z^0,\boldsymbol{\zeta}) \in C}-\alpha(z^0,\boldsymbol{\zeta}_{[0:q-1]}) \\
    & \hspace{40pt}=\max_{(z^0,\boldsymbol{\zeta}) \in C}-\alpha(z^0,\boldsymbol{\zeta}_{[0:q-1]})=: -S<0.
\end{align*}
The above supremum is attained for some $(z^0,\boldsymbol{\zeta}_{[0:q-1]}) \in C$ because $C$ is compact and $\alpha$ is continuous. For any $(z^0,\boldsymbol{\zeta}_{[0:q-1]}) \in C$ we have $V(z^0,\boldsymbol{\zeta}_{[0:q-1]})\geq \delta >0$. The positive definiteness of $V$, following from condition~\eqref{eqn:cond_V_alpha} in Definition~\ref{def:ho-dclf}, implies that $(z^0,\boldsymbol{\zeta}_{[0:q-1]})$ is non-zero and thus the positive definiteness of $\alpha$ implies that the above maximum is strictly less than zero. Since $\boldsymbol{\omega}^{k_n} \in C$ for all $n \in \mathbb{N}$, we have that
\begin{align}\label{eqn:leq-S}
    V(\boldsymbol{\omega}^{k_{n+1}}) - V(\boldsymbol{\omega}^{k_n}) 
    &\leq \sup_{(z^0,\boldsymbol{\zeta}) \in C} \left[\min_{l=1,\dots,m}\hspace{-0.3cm} V(z^l,\boldsymbol{\xi}_{[l:q+l-1]}) - V(z^0,\boldsymbol{\zeta}_{[0:q-1]})\right] \nonumber\\
    &\leq -S \hspace{0.3cm} \forall n \in \mathbb{N}.
\end{align}
To come to a contradiction, we now observe that
\begin{align}\label{eqn:contr_conv_zero}
    V(\boldsymbol{\omega}^{k_N}) = &V(\boldsymbol{\omega}^{k_0}) + \sum_{n=0}^{N-1}V(\boldsymbol{\omega}^{k_{n+1}})-V(\boldsymbol{\omega}^{k_{n}})\nonumber\\
    \leq &V(\boldsymbol{\omega}^{k_0}) - N\cdot S,
\end{align}
where the last inequality follows from equation \eqref{eqn:leq-S}. If $N>V(\boldsymbol{\omega}^{k_0})/S$, then the inequality~\eqref{eqn:contr_conv_zero} implies $V(\boldsymbol{\omega}^{k_N})<0$. This contradicts the positive definiteness of $V$. Hence, our assumption must have been incorrect and there exists a feasible control sequence, which achieves $\lim_{k \to \infty}x^k=0$. 

We now elaborate why under the hypothesis that there does not exist a feasible control sequence which steers $x^k$ to zero as $k$ goes to infinity, $x^k$ can neither be equal to nor get arbitrarily close to zero for some $k \geq k_0$, i.e. it cannot occur that for any $\varepsilon>0$, there exists $k \geq k_0$ with $\|x^{k}\|<\varepsilon$.
If $x^k$ were zero for some $k\geq k_0$, then the feasible zero control would keep the state at zero forever, which would not be in line with the hypothesis. If $x^k$ would get arbitrarily close to zero for some $k\geq k_0$, then there also exists a sequence of feasible controls $\{\boldu^k_{[0:q-1]}\}$ steering $x^k$ to zero thanks to Assumption~\ref{assump:small_control_prop} (Small Control Property): Let $\varepsilon >0$, we show that there exists a feasible control strategy and $K\geq k_0$ such that $\|V(x^k,\boldu^k_{[0:q-1]})\|<\varepsilon$ for all $k \geq K$. This would imply that $\lim_{k \to \infty}V(x^k, \boldu^k_{[0:q-1]}) = 0$. Hence, with the continuity and the positive definiteness of $V$, we would have $\lim_{k \to \infty}x^k=0$, which again would not be in line with the hypothesis. For ease of notation, we make this argument for $m=2$ and consider the interesting case, where some Lyapunov function values do not contribute to the average descent because their weights are zero, e.g. $\sigma_1 = 0$ and $\sigma_2 \geq 2$. Without loss of generality, we may assume that the descent of the Lyapunov function in \adc~always occurs after two time steps. 
By the continuity of $V$ at the origin, we know that for $\varepsilon_0 <\varepsilon$, there exists $\delta_0>0$ such that
\begin{equation*}
    \|x^k,\boldu^k_{[0:q-1]}\| < \delta_0 \Rightarrow \|V(x^k,\boldu^k_{[0:q-1]})\| < \varepsilon_0.
\end{equation*}
More general, we know that due to the continuity of $V$ for $\varepsilon - \varepsilon_0>0$, there exists $\delta>0$ such that
\begin{align*}
    \|(x^k,\boldu^k_{[0:q-1]}) - (x^{k+1},\boldnu^k_{[1:q]})\| < \delta  
    \Rightarrow 
    \|V(x^k,\boldu^k_{[0:q-1]}) - V(x^{k+1},\boldnu^k_{[1:q]})\| < \varepsilon - \varepsilon_0.
\end{align*}
In other words, if 
\begin{align}
    \|x^k, \boldu^k_{[0:q-1]}\| &< \delta_0 \label{eqn:xkukd0}
\intertext{and}
 \|(x^k, \boldu^k_{[0:q-1]}) - (x^{k+1}, \boldnu^{k}_{[1:q]})\| &< \delta, \label{eqn:xkukvkd}
\end{align}
we can conclude that 
\begin{align*}
    \|V(x^{k+1}, \boldnu^{k}_{[1:q]})\| 
    &\leq \|V(x^k, \boldu^k_{[0:q-1]})\| + \|V(x^{k+1}, \boldnu^k_{[1:q]}) - V(x^k, \boldu^k_{[0:q-1]})\|\\
    &< \varepsilon_0 + \varepsilon - \varepsilon_0 = \varepsilon.
\end{align*}
To guarantee that we can find $x^k, \boldu^k_{[0:q-1]}, \boldnu^k_{[0:q+1]}$ that satisfy~\eqref{eqn:xkukd0} and~\eqref{eqn:xkukvkd}, we will use Assumption~\ref{assump:small_control_prop} and the continuity of $f$. First, choose some $0<\varepsilon_f<\delta$. Then by the continuity of $f$, there exists $\delta_f>0$ such that $\|x^k,\nu_0\|<\delta_f$ implies $\|f(x^k,\nu_0)\|<\varepsilon_f$. As the name of  Assumption~\ref{assump:small_control_prop} suggests, we know that for $\varepsilon_c>0$, there exists $\delta_c>0$ such that for all $\|x^k\|<\delta_c$ there exist feasible controls $\boldu^k_{[0:q-1]}$ and $\boldnu^k_{[0:q+1]}$ satisfying \adc~and $\max\{\|\boldu^k_{[0:q-1]}\|,\|\boldnu^k_{[0:q+1]}\|\}< \varepsilon_c$. We can now choose $\varepsilon_c$ small enough, so that $\delta_c + \varepsilon_c + \varepsilon_f + \varepsilon_c < \delta$ and $\delta_c + \varepsilon_c < \min\{\delta_0, \delta_f\}$. Since we can steer the state until $\|x^k\|< \delta_c$ for some $k \geq k_0$, it follows
\begin{align*}
    \|x^k, \boldu^k_{[0:q-1]}\| &\leq \|x^k\| + \|\boldu^k_{[0:q-1]}\| \\
    &< \delta_c +\varepsilon_c < \min\{\delta_0, \delta_f\}
\end{align*}
and
\begin{align*}
    \|(x^k, \boldu^k_{[0:q-1]}) - (x^{k+1}, \boldnu^k_{[1:q]})\| 
    &\leq \|x^k\| + \|\boldu^k_{[0:q-1]}\| + \|x^{k+1}\| +  \|\boldnu^k_{[0:q+1]}\|\\
    &= \|x^k\| + \|\boldu^k_{[0:q-1]}\| + \|f(x^{k}, \nu_0)\| +  \|\boldnu^k_{[0:q+1]}\|\\
    &< \delta_c + \varepsilon_c + \varepsilon_f + \varepsilon_c <\delta.
\end{align*}
In other words, we have found $x^k, \boldu^k_{[0:q-1]}$ and $\boldnu^k_{[0:q+1]}$ satisfying~\eqref{eqn:xkukd0} and~\eqref{eqn:xkukvkd}. This implies that the Lyapunov function value of $(x^{k+1}, \boldnu^k_{[1:q]})$ is less than $\varepsilon$. The Lyapunov function value of $(x^{k+2}, \boldnu^k_{[2:q+1]})$ is naturally less than the Lyapunov function value of $(x^{k}, \boldu^k_{[0:q-1]})$, because this is where the descent occurs. By repeating this argument and using the continuity of $V$ and $f$, as well as Assumption~\ref{assump:small_control_prop} over and over again, we can ultimately show that $V(x^k, \boldu^k_{[0:q-1]})< \varepsilon$ for all $k\geq K$, which finishes the proof.
\end{proof}
The above proof is constructive; it provides us with a control sequence which steers the state to zero. This is summarized in the following proposition.
\begin{proposition}\label{prop:construc_control}
For any $n \in \mathbb{N}$ consider $V(x^{k_{n}}, \mathbf{u}^{k_{n}}_{[0:q-1]})$, where $k_n$ is the index of the subsequence~\eqref{eqn:conv_subsqn}, and let $\boldsymbol{\nu}^{k_n}_{[0:q+m-1]}$ be a control sequence that can be found for $(x^{k_n},\mathbf{u}^{k_n}_{[0:q-1]})$ to satisfy equation~\eqref{eqn:dclf_orderm}. Furthermore, let $l^{k_{n}} \in \{1, \dots,m\}$ be the index of the state and control pair which achieves the descent of the \hodclf~$V$ and thus defines $(x^{k_{n+1}}, \mathbf{u}^{k_{n+1}}_{[0:q-1]})$. Then the control sequence 
\begin{align}\label{eqn:basecontrol}
    [\boldsymbol{\nu}^{k_0}_{[0:l^{k_0}-1]}, \boldsymbol{\nu}^{k_1}_{[0:l^{k_1}-1]}, \dots ,\boldsymbol{\nu}^{k_n}_{[0:l^{k_n}-1]}, \dots]
\end{align}
is feasible and steers the state to zero. Each $\boldsymbol{\nu}^{k_n}_{[0:l^{k_n}-1]}$ has $l^{k_n}$ elements and steers the system to the next state $x^{k_{n+1}}$. 
\end{proposition}
The way the above control sequence was constructed and is used to steer the state gives rise to its time and state dependency. Thus, considering control system~\eqref{eqn:control_system} with the control strategy~\eqref{eqn:basecontrol} results in a time-varying system, explaining the used stability notion.
It remains to prove Theorem~\ref{thm:ho-dclf_lyapstab}. 
\begin{proof}[Proof of Theorem~\ref{thm:ho-dclf_lyapstab}]
This proof is along the lines of classical Lyapunov theory \cite{HK:02}. To ease notation, we  will present it again for the case $m=2$; the general case follows similarly. Theorem~\ref{thm:ho-dclf-stab} already guarantees that any state $x^{k_0}$ can be steered to zero. To obtain asymptotic stability, we still need to show that for all $\varepsilon>0$ and any $k_0 \in \mathbb{N}$, there exists $\bar \delta>0$ such that for any state $\inistatetimev \in X$ with $0 <\|\inistatetimev\|<\bar \delta$, there exists a feasible control strategy which yields $\|x^k\|<\varepsilon$, for all $k\geq k_0$. To this end, fix $\varepsilon> 0$ and $k_0 \in \mathbb{N}$. Due to the radial unboundedness of $V$ by~\eqref{eqn:cond_V_alpha}, we can choose $\mathbf{0} \neq \mathbf{p} \in \mathbb{R}^n \times \mathbb{R}^{p, q}$ such that there exists $r \in (0,\varepsilon)$, where $\|(x,\boldu)\|> r$ implies $V(x,\boldu) > V(\mathbf{p})$, or equivalently,
\begin{align}\label{eqn:rad_unbdd_proof}
    V(x,\boldu) \leq V(\mathbf{p}) \Rightarrow \|(x,\boldu)\|\leq r.
\end{align}
Now define $B_r:=\{(x,\boldu) \in \mathbb{R}^n \times \mathbb{R}^{p, q}: \|(x,\boldu)\|\leq r\}$, $\beta := V(\mathbf{p})>0$ and $\Omega_\beta := \{(x,\boldu) \in B_r: V(x,\boldu) \leq \beta\}$. It follows $\Omega_\beta \subseteq B_r$.
Suppose for now we have 
\begin{align}\label{eqn:Vlessminbeta}
    V(x^{k_0},\boldu^{k_0}_{[0:q-1]})< \min\{0.5\beta \sigma_1,0.5\beta \sigma_2,\beta\}    
\end{align}
for some $x^{k_0}\in X\setminus \{0\}$ and $\boldu^{k_0}_{[0:q-1]} \in \mathbf{U}_{[0:q-1]}(x^{k_0})$. Then we claim that there exists a control strategy such that $V(x^{k}, \mathbf{u}^{k}_{[0:q-1]}) \leq \beta$ for all $k \geq k_0$. Together with equation~\eqref{eqn:rad_unbdd_proof}, this would imply that $(x^k,\boldu^{k}_{[0:q-1]})$ stays in $\Omega_\beta$ for $k\geq k_0$. To prove the claim, let us first consider the case $\sigma_1 \leq \sigma_2$. Clearly, $V(x^{k_0}, \boldu^{k_0}_{[0:q-1]}) < 0.5\beta\sigma_1$. Since $V$ is a \hodclf~there exists a control strategy $\boldsymbol{\nu}_{[0:q+1]}$ such that 
\begin{align*}
    \sigma_2V(x^{k_0+2},\boldsymbol{\nu}_{[2:q+1]}&) + \sigma_1 V(x^{k_0+1},\boldsymbol{\nu}_{[1:q]}) < 2V(x^{k_0},\boldu^{k_0}_{[0:q-1]}) < 2 \frac{\beta\sigma_1}{2} = \beta\sigma_1.
\end{align*}
Hence, $V(x^{k_0+1},\boldnu_{[1:q]}) < \beta$. Similarly, $V(x^{k_0+2},\boldnu_{[2:q+1]})$ $<\beta$.
The case $\sigma_1 > \sigma_2$ follows likewise. By Proposition~\ref{thm:ex_subsqn_stability}, either $V(x^{k_0+1},\boldnu_{[1:q]})$ or $V(x^{k_0+2},\boldnu_{[2:q+1]})$ is less than $V(x^{k_0}, \boldu^{k_0}_{[0:q-1]})$ and determines the index of the descent $\ell \in \{1,2\}$ (compare with proof of Theorem~\ref{thm:ho-dclf-stab}). If $\ell=1$, we have found $x^{k_0+1}$ and $\boldu^{k_0+1}_{[0:q-1]}:=\boldnu_{[1:q]}$, if $\ell=2$, we have also found $x^{k_0+2}$ and $\boldu^{k_0+2}_{[0:q-1]}:=\boldnu_{[2:q+1]}$. The above argument can be repeated with the new state control pair $(x^{k_0+\ell},\boldu^{k_0+\ell}_{[0:q-1]})$. Repeating this argument proves the claim. We hence conclude that  
\[
(x^k, \boldu^k_{[0:q-1]}) \in \Omega_\beta \Rightarrow (x^k, \boldu^k_{[0:q-1]}) \in B_r \quad \forall k \geq k_0. 
\]
Thus, we obtain $\|x^{k}\| \leq \|(x^{k},\boldu^{k}_{[0:q-1]})\| \leq r < \varepsilon$ for all $k\geq k_0$. It remains to guarantee that we can find $x^{k_0}\in X\setminus \{0\}$ and $\boldu^{k_0}_{[0:q-1]} \in \mathbf{U}_{[0:q-1]}(x^{k_0})$ with 
\begin{align*}
V(x^{k_0}, \boldu^{k_0}_{[0:q-1]}) < \min\{0.5\beta\sigma_1,0.5\beta\sigma_2,\beta\}.    
\end{align*}
Due to the fact that $V$ is continuous and $V(0, \mathbf{0})=0$, there exists $\delta>0$ such that $\|(x,\boldu_{[0:q-1]})\| < \delta$ implies 
\begin{align*}
 V(x,\boldu_{[0:q-1]}) < \min\{0.5\beta\sigma_1,0.5\beta\sigma_2,\beta\}. 
\end{align*}
If $q=0$, then the continuity of $V$ is enough, for $q\neq0$ we need to evoke Assumption~\ref{assump:small_control_prop}. It guarantees that for all $\bar \varepsilon>0$, there exists $\bar \delta>0$ such that for all $0<\|x^{k_0}\|<\bar \delta$ there exists $\boldu^{k_0}_{[0:q-1]} \in \mathbf{U}_{[0:q-1]}(x^{k_0})$ with $\|\boldu^{k_0}_{[0:q-1]}\|<~\bar\varepsilon$ satisfying the inequalities~\eqref{eqn:cond_V_alpha} and~\eqref{eqn:dclf_orderm}. Thus, for small enough $\bar \varepsilon$ we have $\bar \varepsilon + \bar \delta< \delta$ and obtain that for any $x^{k_0} \in X$ with $0<\|x^{k_0}\|<\bar \delta$ there exists $\boldu^{k_0}_{[0:q-1]} \in~\mathbf{U}_{[0:q-1]}(x^{k_0})$ such that
\[ 
\|(x^{k_0}, \boldu^{k_0}_{[0:q-1]})\| \leq \|x^{k_0}\| + \|\boldu^{k_0}_{[0:q-1]}\| < \bar \delta +\bar \varepsilon < \delta.
\]
We can deduce that the Lyapunov function value of $(x^{k_0},\boldu^{k_0}_{[0:q-1]})$ satisfies the desired bound~\eqref{eqn:Vlessminbeta}. In summary, if $\inistatetimev \in X$ satisfies $0<\|\inistatetimev\|<\bar \delta$, then $\|x^k\|<\varepsilon$ for all $k \geq k_0$. This completes the proof. 
\end{proof}

The proofs of Theorem~\ref{thm:algo_convtozero} and~\ref{thm:algo_conv} can be found in Appendix~\ref{sec:appendix}.

\section{Applications to nonholonomic systems}\label{sec:application_nonholo_sys}
In this section, we demonstrate the performance of our algorithm for a nonholonomic system. In particular, we consider a variation of the classical Brockett integrator~\cite{RB:83} in discrete-time. Let $x^0 \in \mathbb{R}^4$ be the initial state, $h=0.1$ the sampling time and $u^{j}=[u^j_1\ u^j_2]^T \in \mathbb{R}^2$ the control input at time step $j$. The dynamics of the state $x^{j+1} = [x^{j+1}_1\ x^{j+1}_2\ x^{j+1}_3\ x^{j+1}_4]^T \in \mathbb{R}^4$ at time step $j+1$ are given by 
\begin{align}
\begin{split}\label{sys:plus_000abs}
        x^{j+1} &= x^{j} + h\hspace{-0.1cm}
        \begin{bmatrix}
        1\\ 
        0\\
        -x_2^j\\ 
        x_3^j
        \end{bmatrix}
        u^j_1 + h\hspace{-0.1cm}
        \begin{bmatrix}
        0\\ 
        1\\
        x^j_1\\
        x_2^j
        \end{bmatrix}
        u^j_2 +h\hspace{-0.1cm}
        \begin{bmatrix}
        0\\ 
        0\\
        0\\
        |x_4^j|
        \end{bmatrix}\hspace{-0.13cm}.
\end{split}
\end{align}
It is worth pointing out that in \eqref{sys:plus_000abs} we have added the last term to ``break the symmetry'', inducing the need for a richer class of control inputs. This system is difficult to stabilize because it does not satisfy Brockett's necessary condition as stated in the next proposition. Therefore, there does \textit{not} exist a smooth state feedback control law which renders $(0,0)$ locally asymptotically stable.
\begin{proposition}\label{thm:1st_system_fails_brockett}
The control system~\eqref{sys:plus_000abs} does not satisfy Brockett's necessary condition.
\end{proposition}
The proof of this result can be found in Appendix~\ref{sec:appendixB}. The difficulty in stabilizing the above system makes it suitable for our scheme. In particular, standard MPC schemes with quadratic costs are known to encounter issues in such settings~\cite{MM-KW:17}. Before we look at the simulation results, we will first propose a \hodclf~for system~\eqref{sys:plus_000abs} in the next section.
\subsection{Choice of generalized control Lyapunov function}
We take the function $V: \mathbb{R}^n \to \mathbb{R}, x \mapsto \|x\|^2$ as a \hodclf~candidate, corresponding to the case $q=0$. Although system~\eqref{sys:plus_000abs} does \textit{not} satisfy Brockett's necessary condition, we are able to decrease the state in multiple steps with respect to $\|\cdot\|^2$. The multiplicity of steps is the key feature that enables the descent. In particular, let $e_i \in \mathbb{R}^4$ be the unit vector with the $i$th entry of one. Suppose that initially $x^0 = e_4$. Hence, we wish to decrease the fourth component of the state. According to our dynamics, one way to do so is to set $u^0_1=0$ and to choose $u^0_2=5$ so that we increase $x^1_2$: 
\begin{align*}
    x^1 &= e_4
        + 0 e_1
        + 0.5e_2
        + 0.1e_4 = \begin{bmatrix}
        0 & 0.5 & 0 & 1.1 \end{bmatrix}^T\hspace{-0.3cm}.
    \intertext{This increases the squared norm of the state after one step, but in the second step we are able to reduce $x^2_4$ through $x_2^1$. By choosing $u^1_1=0$ and $u^1_2=-u^0_2$, we reverse the previous increase of the second component of the state: }
    x^2 &= \begin{bmatrix}
        0 & 0.5 &0 & 1.1 \end{bmatrix}^T
        + 0\begin{bmatrix}
        1& 0 & -0.5 &0 \end{bmatrix}^T - 0.5 \begin{bmatrix}
        0& 1&0 &0.5 \end{bmatrix}^T\hspace{-0.1cm}
        + 0.11 e_4 = 0.96e_4.
\end{align*}
Note that $\|x^0\|^2=1, \|x^1\|^2 = 1.46$ and $\|x^2\|^2=0.9216$, demonstrating a descent in two steps. The above control strategy obviously strongly depends on the initial state. These computations motivate the approach of using multiple steps to decrease a Lyapunov function value. Our analytical computations match Figure~\ref{fig:desc_in_10_steps}, where a control strategy found by simulations decreases the squared norm of the states in ten steps.

In summary, the above observations demonstrate that the function value of $\|\cdot\|^2$ can be decreased in multiple steps along trajectories of system~\eqref{sys:plus_000abs}. Therefore, we choose $V(x)= \|x\|^2$ as the \hodclf~for system~\eqref{sys:plus_000abs}.   
\begin{figure}
\centering
    \includegraphics[width=0.6\linewidth]{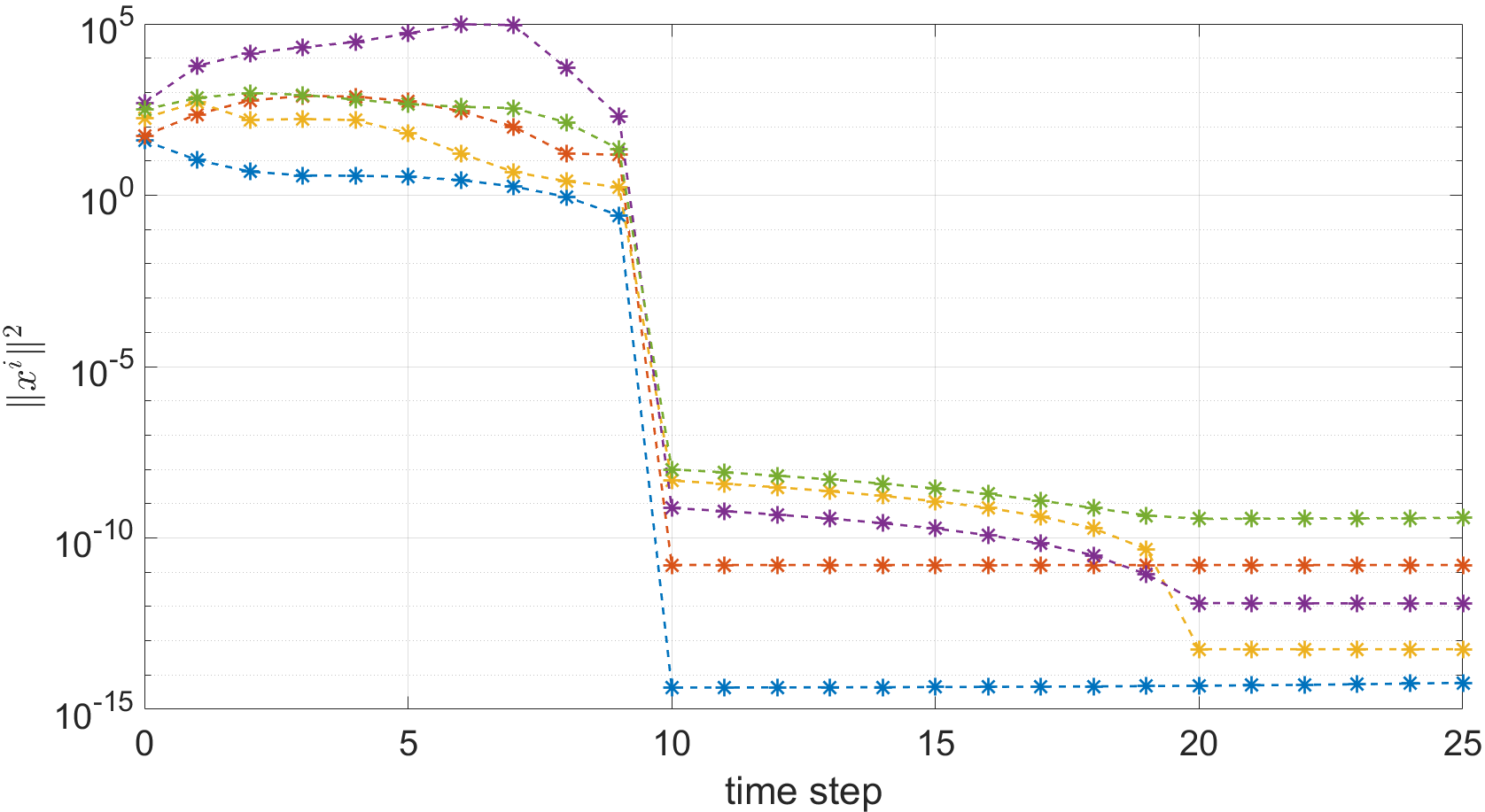}
    \caption{The function value of $\|x\|^2$ along trajectories of system~\eqref{sys:plus_000abs}. The function value of five initial states, which were chosen randomly from a normal distribution, was successfully decreased in ten steps.}
    \label{fig:desc_in_10_steps}
\end{figure}
\subsection{Simulation results}\label{sec:simu_results}
Let us investigate Algorithm \ref{algo:ho-dclf}, when applied to the following optimal control problem.
\begin{problem}\label{prob:simulation}
\begin{align*}
\begin{aligned}
    \min &\sum_{j=0}^{N_p-1} \|x^j\|^2 + 5\,\|u^{j}\|^2\\
    \mathrm{s.t.} &\mathrm{~\eqref{sys:plus_000abs}~ holds~with~}  x^0 = x\\
    &\frac{5.5}{10}\Big(\|x^{6}\|^2 + \|x^{5}\|^2 + \|x^{4}\|^2 +\|x^{3}\|^2\Big) - \|x^{0}\|^2  \leq -\varepsilon\|x^0\|^4
        \end{aligned}
\end{align*}
\end{problem}
We take the prediction horizon to be $N_p=10$. The initial state is set to be $[1\ 2\ 3\ 5]^T$ and we use
\begin{align*}
\begin{bmatrix}
    1 & 1&  \dots & 1 \\1 & 1 & \dots & 1 
    \end{bmatrix} \in \mathbb{R}^{2,10}
\end{align*}
as the initial control. As discussed, we choose $V(x) = \|x\|^2$ as the \hodclf. In step four of Algorithm~\ref{algo:ho-dclf}, we find an index $\ell_{decr}\in \{1,\dots,m\}$ which achieves a descent for this \hodclf. We do so by choosing the index $\ell_{decr}$ that corresponds to the minimal predicted Lyapunov function value $V(x^{\ell_{decr}})$. To achieve a descent, we choose $\alpha(x) = \varepsilon\|x\|^4$, where $\varepsilon = 10^{-5}$. Although only the weights of the sixth, fifth, fourth and third Lyapunov function value are non-zero, we choose $m=N_p=10$ for \adc. This means, in every iteration of Algorithm \ref{algo:ho-dclf} we can implement up to ten steps. In our implementation {\tt fmincon} from MATLAB is used to solve Problem~\ref{prob:simulation}.

The solution to Problem~\ref{prob:simulation} computed by flexible-step MPC is shown in Figure \ref{fig:control_10adc}.
\begin{figure}
\includegraphics[width=0.6\linewidth]{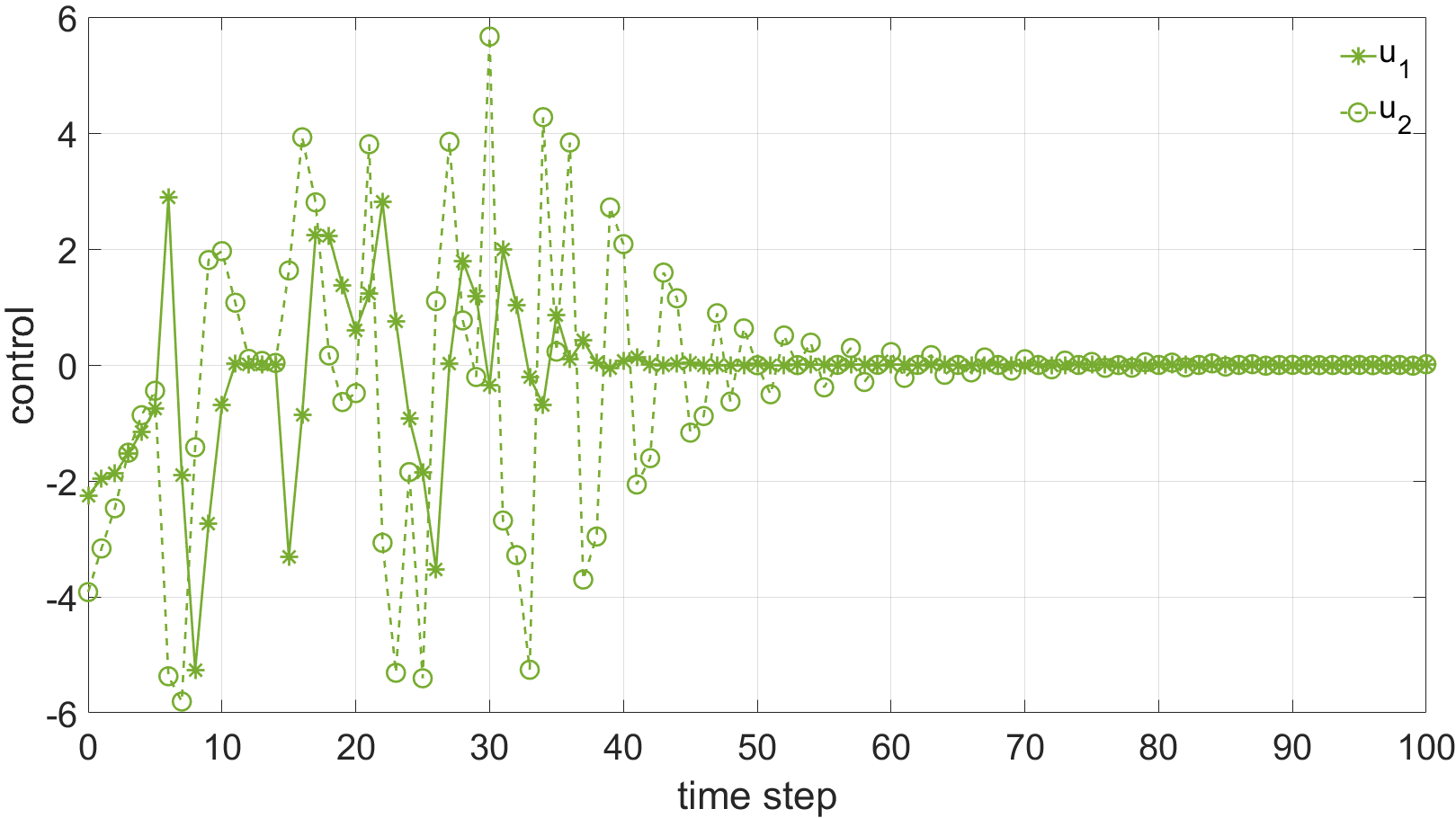}
\caption{Solution to Problem~\ref{prob:simulation}}
\label{fig:control_10adc}
\end{figure}
The resulting trajectories of the states are shown in Figure \ref{fig:state_10adc}. We observe that after a transient phase, the states successfully converge to zero.
\begin{figure*}
\begin{subfigure}[b]{0.495\textwidth}
\includegraphics[width=1.0\linewidth]{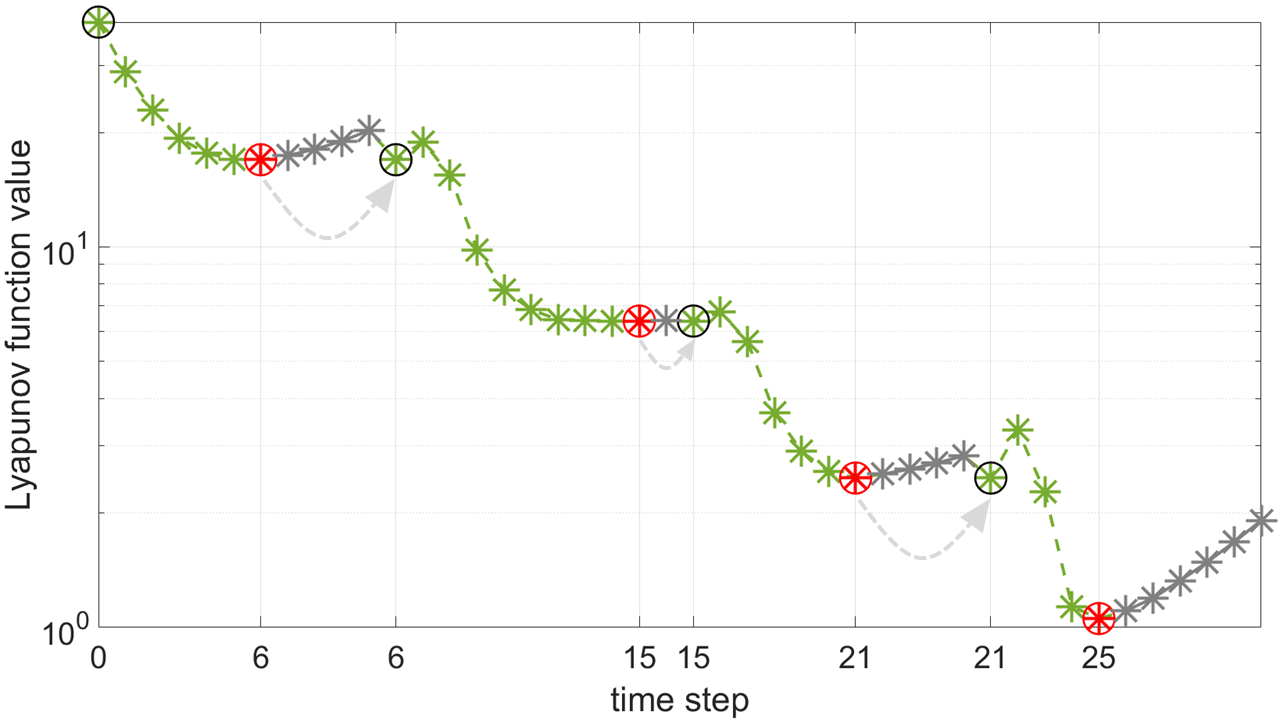}
\caption{First period}
\label{fig:lyap_fcn_10adc_detail}
\end{subfigure}
\begin{subfigure}[b]{0.495\textwidth}
\includegraphics[width=1.0\linewidth]{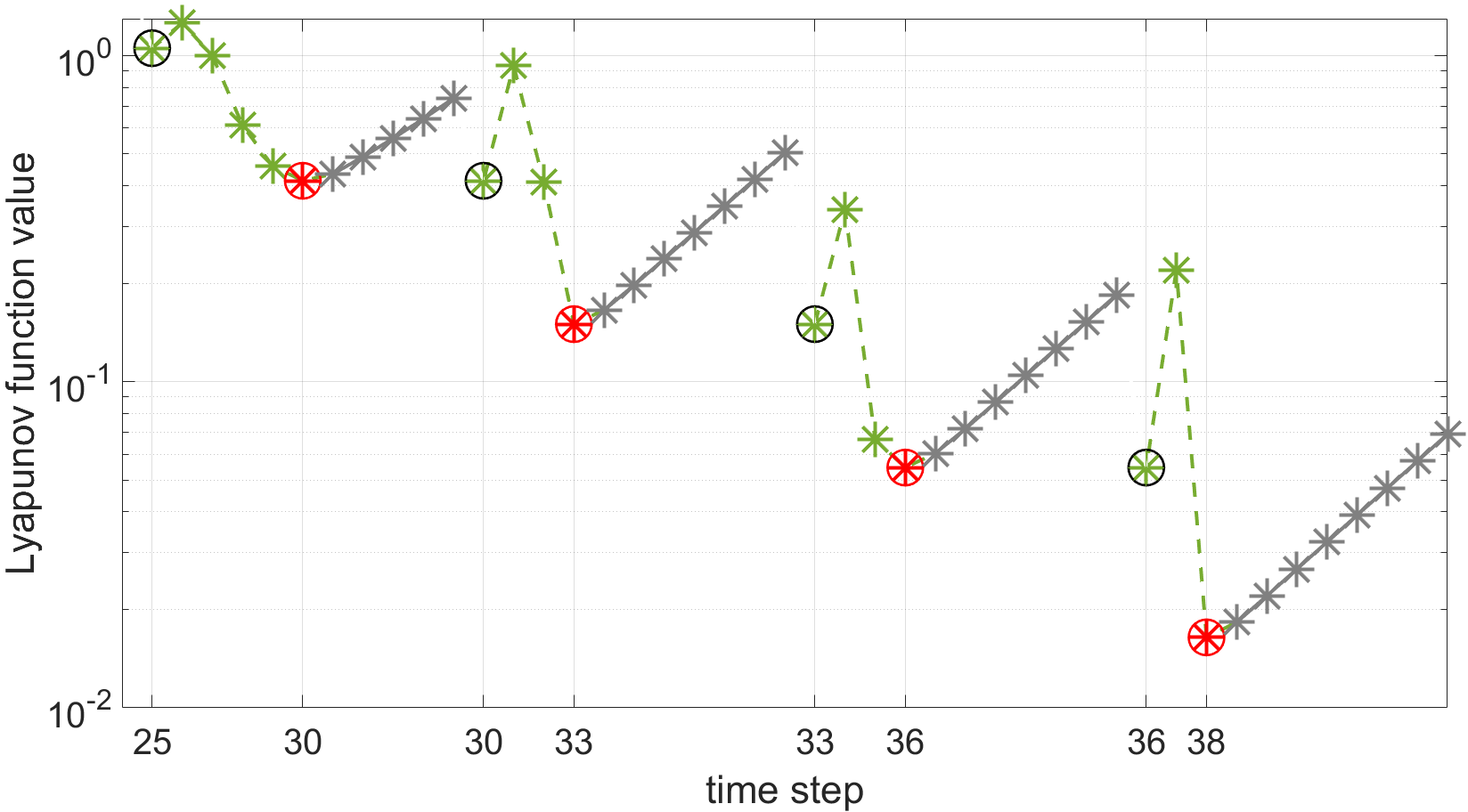}
\caption{Second period}
\end{subfigure}
\caption{Two periods of the Lyapunov function value according to the solution of Problem~\ref{prob:simulation}. The Lyapunov function value of the initial state in the current iteration is circled in black. It is followed by the $m=10$ Lyapunov function values over the prediction horizon. The minimal Lyapunov function value of each iteration is highlighted in red. The dashed arrow symbolizes that the state, which achieves the minimum, will be the initial state for the next iteration, while the gray Lyapunov function values following the minimum will be discarded. This is why the time axis is unconventional; the time does not continue until the new iteration with the just declared initial state starts.}
\label{fig:lyap_fcn_10adc}
\end{figure*}
\begin{figure}
\includegraphics[width=0.6\linewidth]{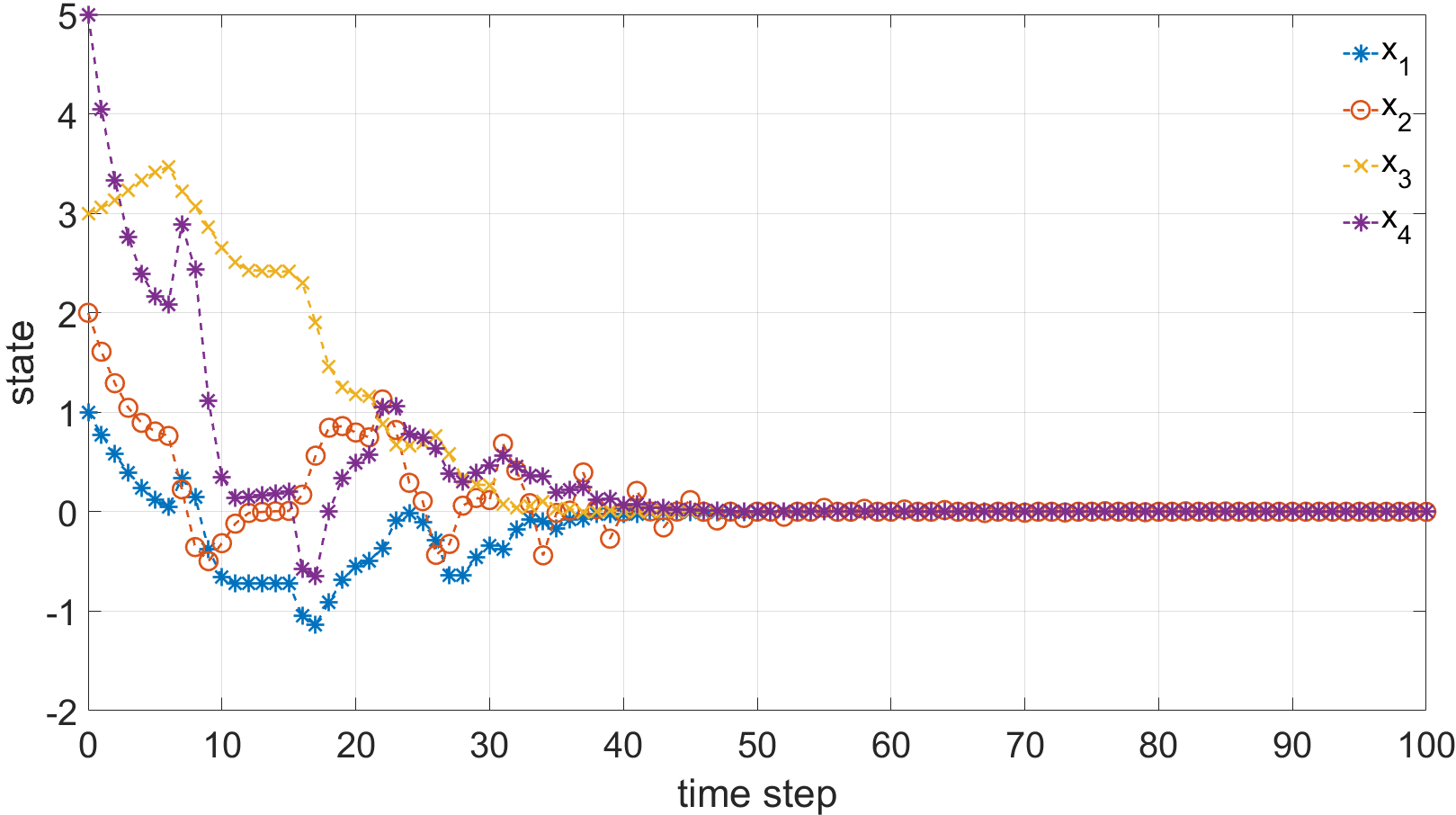}
\caption{State trajectories according to the solution of Problem~\ref{prob:simulation}}
\label{fig:state_10adc}
\end{figure}
In Figure \ref{fig:lyap_fcn_10adc}, we can see the Lyapunov function values for two selected periods; the outcome of \adc~can be clearly recognized in this figure. The Lyapunov function value of the initial state of the current iteration is circled in black. The red star indicates the minimum of the Lyapunov function values in every iteration. The Lyapunov function values following the minimum are gray because they will be discarded in our implementation. Instead, the controls will be implemented until the minimum occurs and the resulting state will be the new initial state. Displaying the Lyapunov function values over the whole prediction horizon in each optimization instance results in an unconventional time axis. For example, in Figure~\ref{fig:lyap_fcn_10adc_detail} the minimal Lyapunov function value of the first optimization instance of Problem~\ref{prob:ho-dclf} occurs at time step six and the time does not continue until after the next black circle which indicates the Lyapunov function value of the new initial state. Altogether, we can see in Figure \ref{fig:lyap_fcn_10adc} how the Lyapunov function values evolve and how various optimization instances present a different minimal index. A summary of the implemented steps at each optimization instance of Problem~\ref{prob:simulation} is given in Figure \ref{fig:loc_descent_10adc}. This figure can be interpreted as follows: The solution of the first optimization instance of Problem~\ref{prob:simulation} results in the implementation of six steps. The second optimization, which starts at $k=6$, yields the implementation of nine steps and so on.     
\begin{figure}
\includegraphics[width=0.6\linewidth]{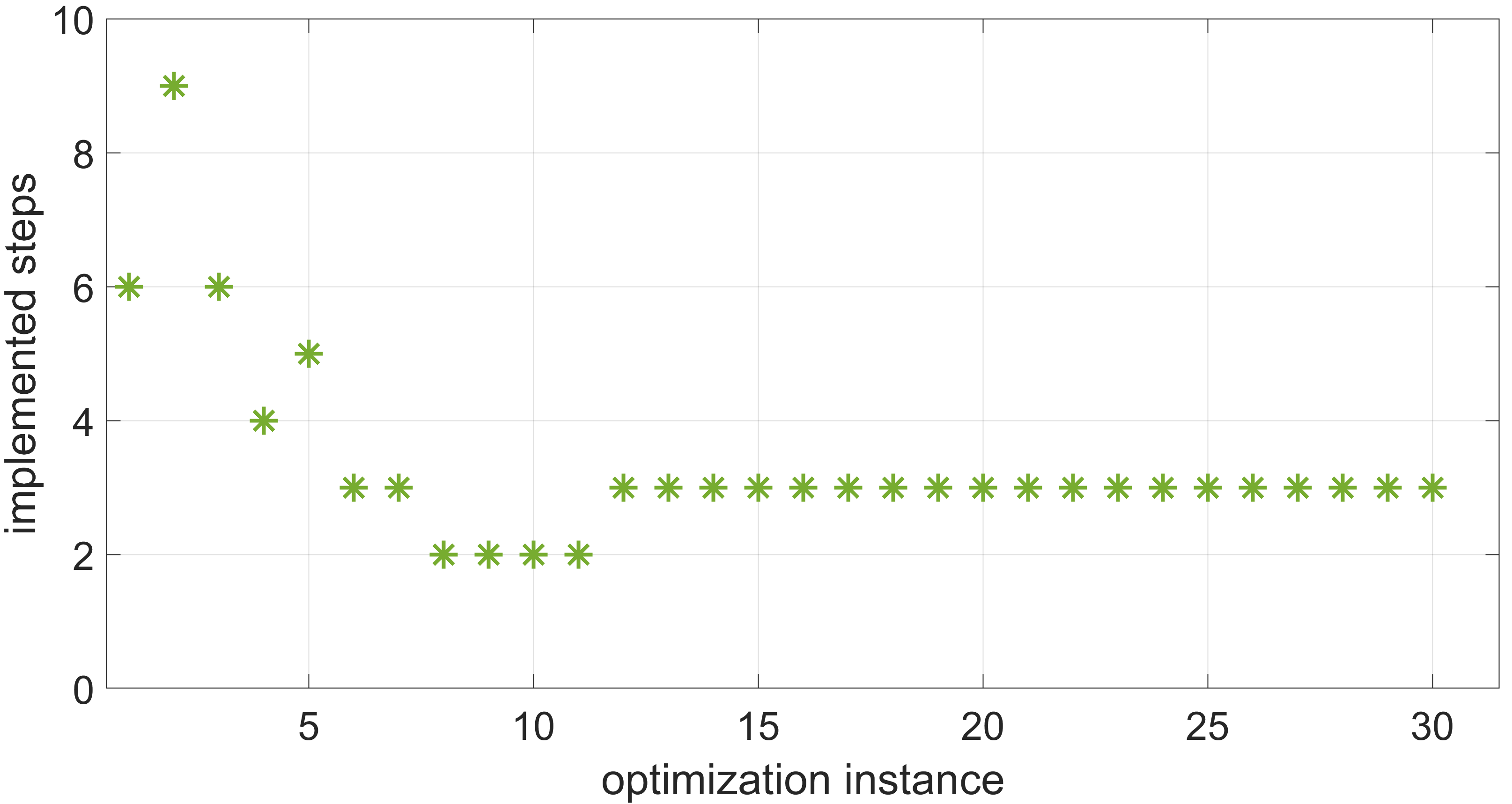}
\caption{In each optimization instance, Problem \ref{prob:simulation} is solved. The solution yields the number of implemented steps, depicted on the vertical axis. Note that the sum of implemented steps of the optimization instances corresponds to the time step of the implementation. }
\label{fig:loc_descent_10adc}
\end{figure}
Note that the algorithm is truly making use of the flexible steps in the transient phase of the state trajectory (optimization instances one to thirteen).

We will compare our solution of Problem \ref{prob:simulation} to standard MPC with a terminal cost, weighted by $\gamma > 0$, and the prediction horizon $N_p=10$, see Problem~\ref{prob:simulation_clMPC}, where a fixed number of ten steps is always implemented. 
\begin{problem}\label{prob:simulation_clMPC}
\begin{align*}
\begin{aligned}
    \min &\sum_{j=0}^{N_p-1} \|x^j\|^2 + 5\,\|u^{j}\|^2 + \gamma \|x^{10}\|^2\\
    \mathrm{s.t.} &\mathrm{~\eqref{sys:plus_000abs}~ holds~with~}  x^0 = x
        \end{aligned} 
\end{align*}
\end{problem}
We have verified through simulations that the following observations remain qualitatively the same with other numbers of implemented steps. As a first choice for the weight of the terminal cost, we take $\gamma = 22$. The state trajectories resulting from the solution to Problem~\ref{prob:simulation_clMPC} computed by standard MPC with $\gamma=22$ are depicted in Figure \ref{fig:states_termcost_22(xN)}, where we do not observe convergence to zero.
\begin{figure}
\includegraphics[width=0.6\linewidth]{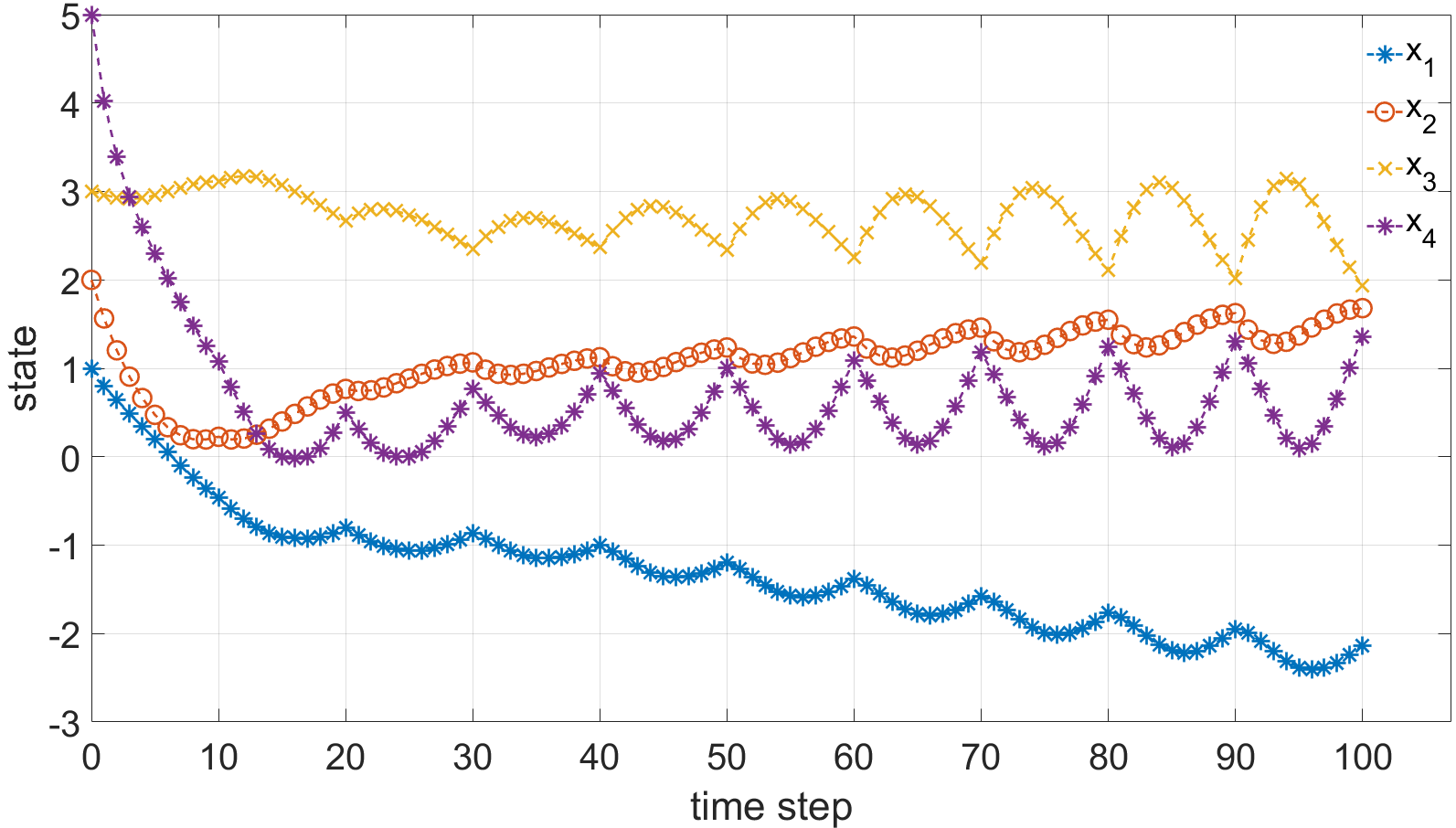}
\caption{State trajectories according to the solution of Problem~\ref{prob:simulation_clMPC} with $\gamma =22$}
\label{fig:states_termcost_22(xN)}
\end{figure}
This is why we increased the weight of the terminal cost, while keeping the same Lyapunov function $V(x)=\|x\|^2$. Increasing the weight does help, as seen in Figure \ref{fig:states_termcost_480(xN)}, but $x_2$ and $x_4$ keep oscillating.
\begin{figure}
\includegraphics[width=0.6\linewidth]{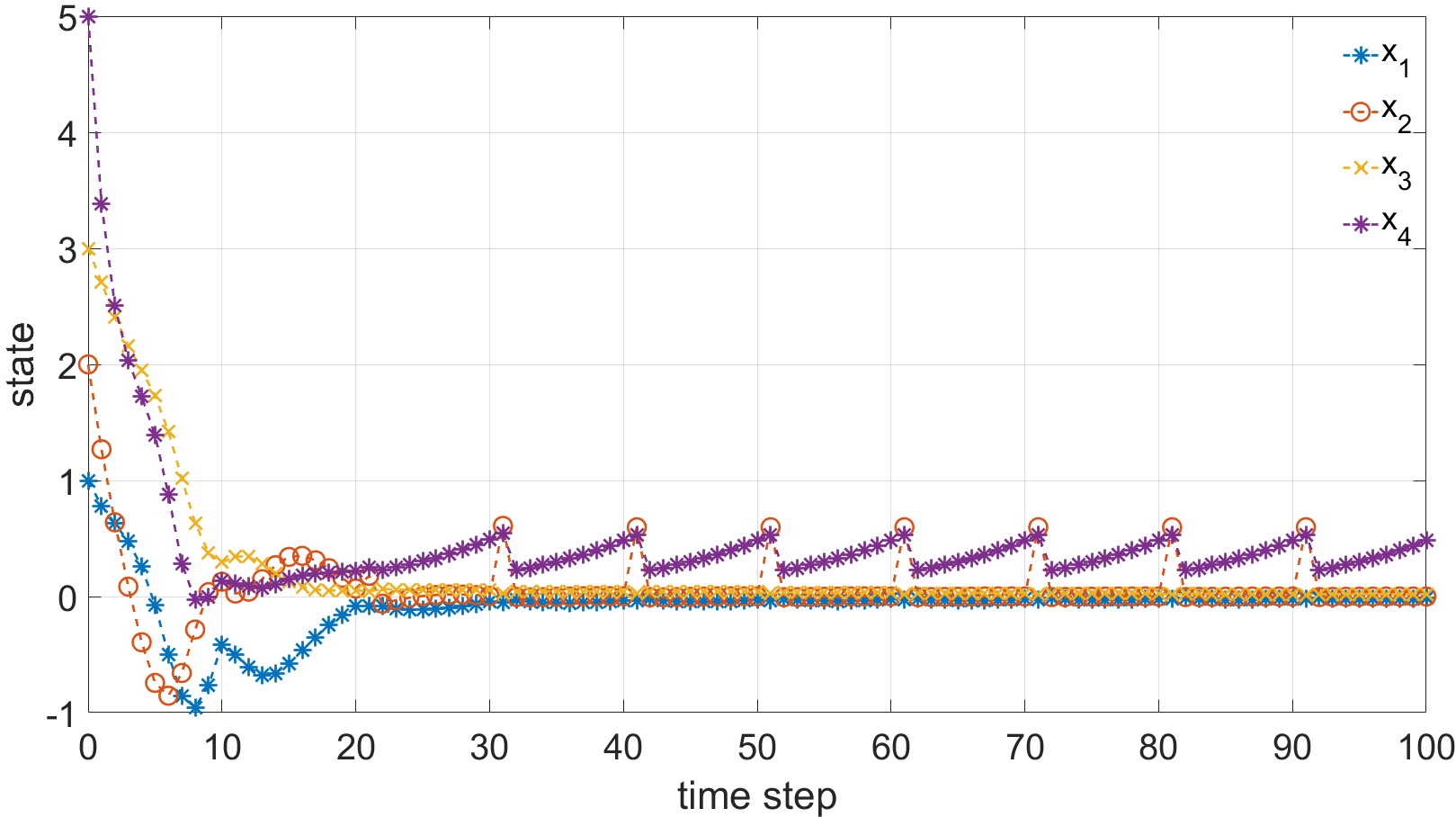}
\caption{State trajectories according to the solution of Problem~\ref{prob:simulation_clMPC} with $\gamma =480$}
\label{fig:states_termcost_480(xN)}
\end{figure}
We can increase the weight even further, see Figure \ref{fig:states_termcost_480(xN)}, but the oscillations remain.
\begin{figure}
\includegraphics[width=0.6\linewidth]{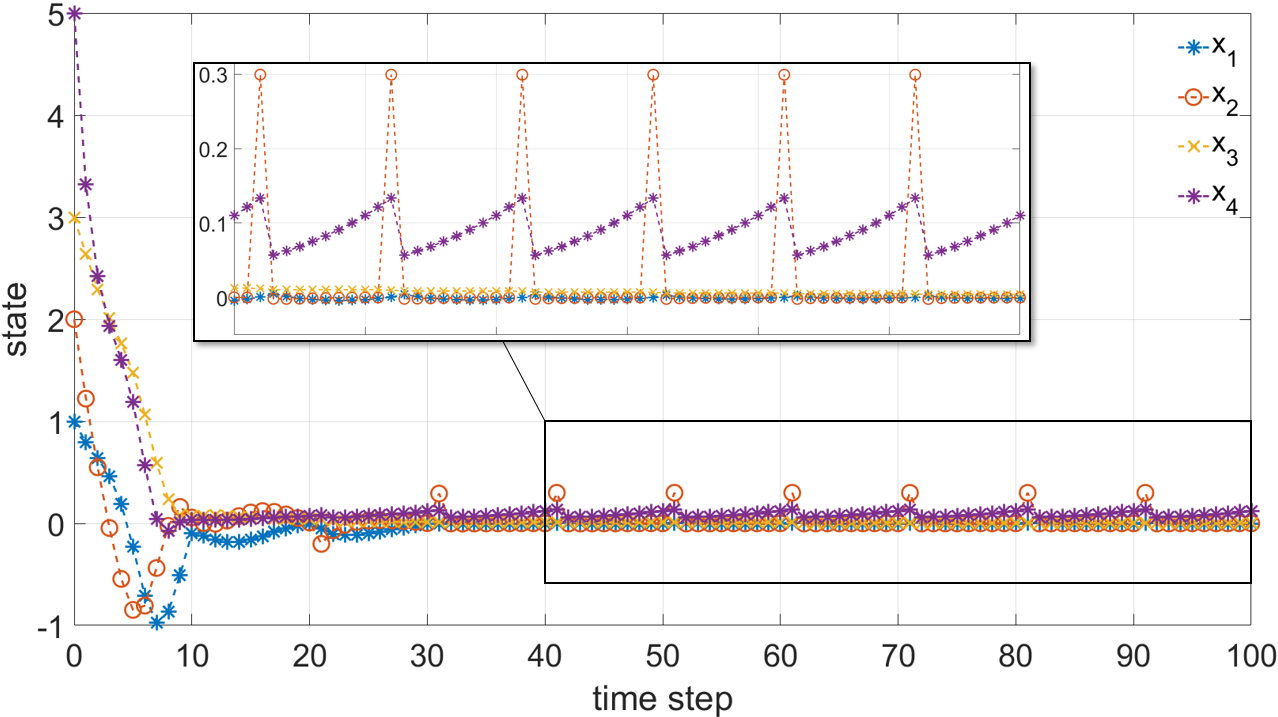}
\caption{State trajectories according to the solution of Problem~\ref{prob:simulation_clMPC} with $\gamma =1920$}
\label{fig:states_termcost_1920(xN)}
\end{figure}
Even with very high weights in the terminal cost, as in Figure~\ref{fig:states_termcost_1920(xN)}, we never reach convergence of the states to zero. The weight of the terminal cost only damps the oscillations of $x_2$ and $x_4$. In summary, the desired stability behavior was not achieved by increasing the weight of the terminal cost. Apparently, the standard approach involves more challenges in terms of terminal ingredients than our approach to obtain convergence of the state to zero, even though the same \hodclf~of $\|\cdot\|^2$ was used. Finally, we compare the total function value of the standard MPC scheme to our approach in Figure \ref{fig:total_fcn_val_compare}. The total function value is the value of $\sum_{j=0}^{k-1}\|x(j)\|^2 + 5\|u(j)\|^2 $, where $k$ is the current time step. The oscillations also have consequences here: The total function value always slightly increases because the states have not converged to zero. 

\begin{figure}
\includegraphics[width=0.6\linewidth]{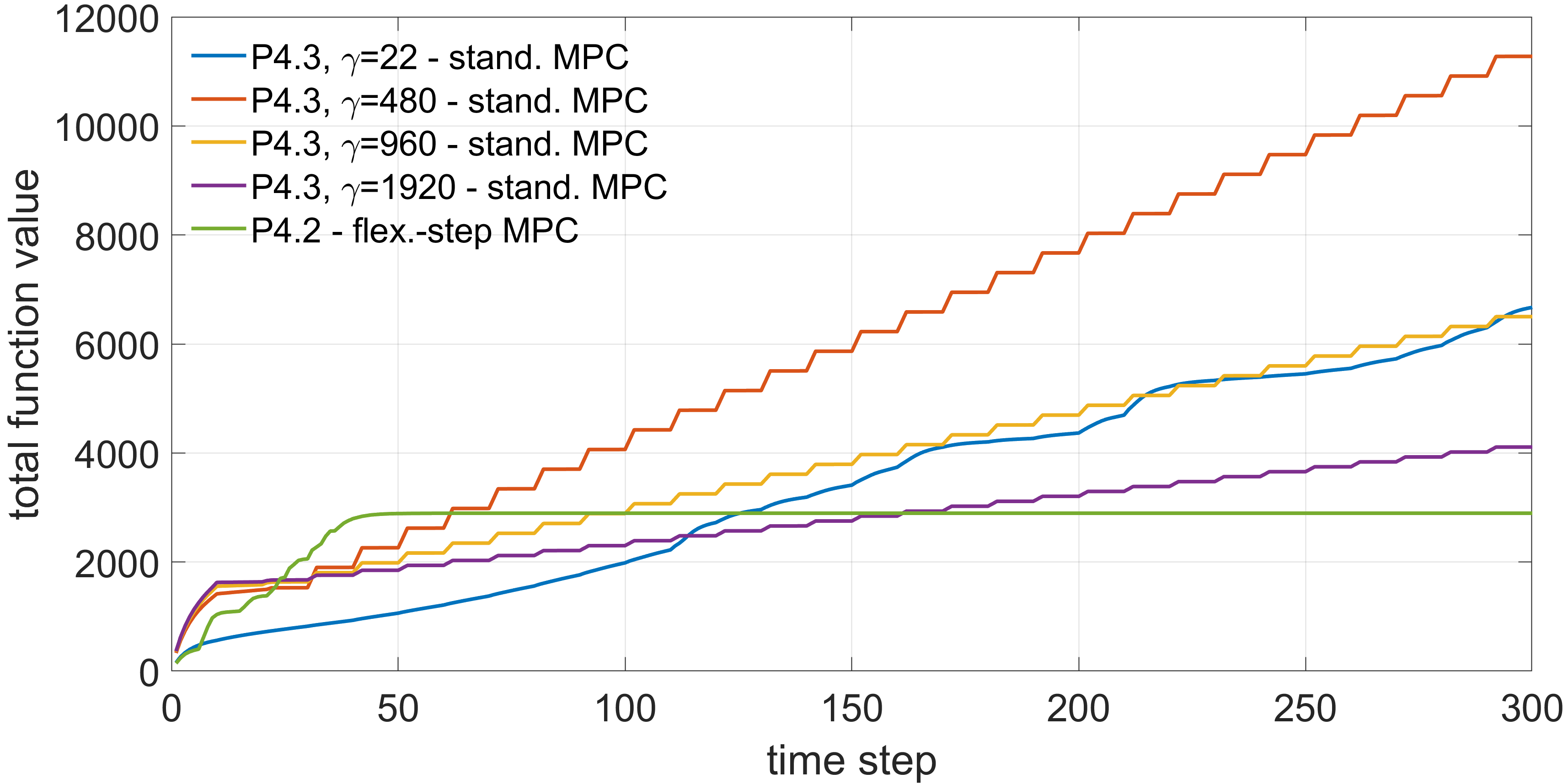}
\caption{The total function value, i.e. $\sum_{j=0}^{k-1}\|x(j)\|^2 + 5\|u(j)\|^2 $, where $k$ is the current time step, is displayed. The total function values according to the solution of Problem~\ref{prob:simulation} by flexible-step MPC and the solution of Problem~\ref{prob:simulation_clMPC} with different weights $\gamma$ by standard MPC are compared.}
\label{fig:total_fcn_val_compare}
\end{figure}

\section{Conclusion}\label{sec:conclusion}
We generalized the idea of higher-order and non-monotonic Lyapunov functions by introducing \hodclf s. Through the idea of adding the crucial condition~\eqref{eqn:dclf_orderm} from the definition of \hodclf s as a constraint, we obtained a novel MPC scheme. This constraint alone facilitated the flexibility in the number of implemented steps in each iteration of our proposed scheme. We demonstrated the benefits of this flexibility through an application example. 

There are several theoretical aspects of our work, where we would like to make further advancement in the future. One goal is to prove recursive feasibility for $q\neq N_p$, i.e. we would like to extend the result of Theorem~\ref{thm:rec_feas}. Another goal is to provide a computationally efficient way for computing the terminal set $X^{N_p}$ of Problem~\ref{prob:ho-dclf}. We would also like to find a way to bypass Assumption~\ref{assump:small_control_prop} (Small Control Property). The current choice of $\boldu^{*-}_{[0:q-1]}$ in Algorithm~\ref{algo:ho-dclf} guarantees a chain of inequalities, similar to~\eqref{eqn:Vstrictineqs}. We believe that there are other possible choices (see Remark~\ref{rkm:changeustarminus}), which we would like to investigate. Furthermore, we would like to explore the idea of using \hodclf s in the terminal cost, as opposed to incorporating them in a constraint, as well as connecting \hodclf s to classical Lyapunov functions.

Our simulations showed that variations of the classical Brockett integrator provide an area of application for this novel implementation. In that sense, it would be interesting to study the performance of our scheme for applications like mobile robots, where the underlying dynamics are nonholonomic, in the future. Moreover, it would be interesting to apply our flexible-step implementation to large scale systems, where implementing more than one step could make a significant difference for the running time. 

Other promising directions are reference tracking and generalizing our framework to systems subject to disturbances or to systems with an unknown model. It is especially attractive for the latter case; the relaxed stability requirement of \adc~enables more exploration compared to the classical one. Lastly, we would like to rigorously study the performance of our novel implementation. 


\bibliographystyle{plain}
\bibliography{alias,Group-bib,Main-add,Main}

\appendix
\section{Proofs of Theorem~\ref{thm:algo_convtozero} and~\ref{thm:algo_conv}}\label{sec:appendix}
For the sake of completeness, we present the proofs of Theorem~\ref{thm:algo_convtozero} and~\ref{thm:algo_conv}. Recall that we are considering system~\eqref{eqn:control_system} with the time-varying MPC controller $c^k_{mpc}$ resulting in the time-varying system~\eqref{eqn:flexstepdynamics}. This in turn has to be studied with the appropriate stability notion. 
\begin{proof}[Proof of Theorem~\ref{thm:algo_convtozero}]
Fix $k_0 \in \mathbb{N}$ and $x(k_0) \in X$. We may assume again that $x(k_0)$ is unequal to zero. Suppose, by contradiction, that the control law generated by Algorithm~\ref{algo:ho-dclf} does not steer $x(k)$ to zero as $k$ goes to infinity. This implies that $x(k)$ can neither be equal to nor get arbitrarily close to zero for some $k \geq \initimetimev$, see proof of Theorem~\ref{thm:ho-dclf-stab}. As before, our goal is to define a subsequence $\{(x(k_n), \boldu^{k_n}_{[0:q-1]})\}_{n=0}^{\infty}$, where the Lyapunov function values evaluated at these states and controls strictly decrease. To this end, consider $x(k_0)$ as the current state of the flexible-step MPC scheme, which defines the initial state of Problem~\ref{prob:ho-dclf}, $x^0=x(k_0)$, and let $\boldu^{*-}_{[0:q-1]} \in \mathbf{U}_{[0:q-1]}(x(\initimetimev))$ be a fixed, but arbitrary, control sequence. Let $\boldu^*(x(\initimetimev))_{[0:N-1]}$ be the control sequence generated by Algorithm~\ref{algo:ho-dclf} in step three, where $N=\max\{q+m,N_p\}$. Since \adc~is satisfied, Algorithm~\ref{algo:ho-dclf} can find an index $1 \leq \ell_{decr} \leq m$ in step four for which
\begin{align*}
    &V(x^{*\ell_{decr}},\boldu^*(x(\initimetimev))_{[\ell_{decr}:q+\ell_{decr}-1]}) < V(x(\initimetimev),\boldu^{*-}_{[0:q-1]}),   
\end{align*}
where $\boldu^*(x(\initimetimev))_{[\ell_{decr}:q+\ell_{decr}-1]} \in \mathbf{U}_{[0:q-1]}(x^{*\ell_{decr}})$ (compare with Proposition~\ref{thm:ex_subsqn_stability}). In this proof, we choose the index with the greatest descent in the above inequality as $\ell_{decr}$.
Following the flexible-step MPC scheme, the control sequence $\boldu^*(x(\initimetimev))_{[0:\ell_{decr}-1]}$ is implemented and as a result, we can define the first two members of our subsequence:
\begin{align*}
    \boldsymbol{\omega}^{k_0}&:=(x(\initimetimev),\boldu^{*-}_{[0:q-1]})\\
    \boldsymbol{\omega}^{k_1}&:= (x(\initimetimev+ \ell_{decr}),\boldu^*(x(\initimetimev))_{[\ell_{decr}:q+\ell_{decr}-1]}).
\end{align*}
To define the rest of the subsequence, we proceed as follows. Suppose $\boldsymbol{\omega}^{k_n}$ has been defined. Then Algorithm~\ref{algo:ho-dclf} with the current state $x(k_n)$ produces again a control sequence $\boldu^*(x(k_n))_{[0:N-1]}$ such that for some $1 \leq \ell_{decr}\leq m$ the Lyapunov function value is strictly below $V(\boldsymbol{\omega}^{k_n})$ and $\boldu^*(x(k_n))_{[\ell_{decr}:q+\ell_{decr}-1]} \in \mathbf{U}_{[0:q-1]}(x^{*\ell_{decr}})$. We can define $\boldsymbol{\omega}^{k_{n+1}}$ accordingly
\begin{align*}
    \boldsymbol{\omega}^{k_{n+1}} &:= (x(k_n+\ell_{decr}),\boldu^*(x(k_n))_{[\ell_{decr}:q+\ell_{decr}-1]}).
\end{align*}
Along the subsequence $\{\boldsymbol{\omega}^{k_n}\}_{n=0}^{\infty}$, we obtain the following inequalities
\begin{align}\label{eqn:Vstrictineqs_mpc}
    V(\boldsymbol{\omega}^{k_0}) > V(\boldsymbol{\omega}^{k_1})&> \dots > V(\boldsymbol{\omega}^{k_n}) > \dots \geq 0.
\end{align}
Like in the proof of Theorem~\ref{thm:ho-dclf-stab}, the sequence $\{V(\boldsymbol{\omega}^{k_n})\}_{n=0}^\infty$ converges to some $ \delta >0$ as $n \to \infty$, which allows us to define the set $C$ once more. Similarly, we conclude that
\begin{align*}
    &\sup_{(z^0,\boldsymbol{\zeta}) \in C} \left[\min_{l=1,\dots,m}\hspace{-0.1cm} V(z^l,\boldsymbol{\xi}_{[l:q+l-1]}) - V(z^0,\boldsymbol{\zeta}_{[0:q-1]})\right] \\
    & \hspace{40pt}=\max_{(z^0,\boldsymbol{\zeta}) \in C}-\alpha(z^0,\boldsymbol{\zeta}_{[0:q-1]})=: -S<0
    \intertext{and}
    &\hspace{10pt}V(\boldsymbol{\omega}^{k_{n+1}}) - V(\boldsymbol{\omega}^{k_n}) \leq -S \quad \forall n \in \mathbb{N},
\end{align*}
where $\{\boldsymbol{\omega}^{k_n}\}_{n=0}^{\infty}$ are produced by Algorithm~\ref{algo:ho-dclf}, as described above. We obtain the same contradiction as before by observing
\begin{align}\label{eqn:contr_conv_zero_wMPC}
    V(\boldsymbol{\omega}^{k_{\bar N}}) = &V(\boldsymbol{\omega}^{k_0}) + \sum_{n=0}^{\bar N-1}V(\boldsymbol{\omega}^{k_{n+1}})-V(\boldsymbol{\omega}^{k_{n}})\nonumber\\
    \leq &V(\boldsymbol{\omega}^{k_0}) - \bar N\cdot S.
\end{align}
Recall that $V$ is positive definite by assumption due to condition~\eqref{eqn:cond_V_alpha} in Definition~\ref{def:ho-dclf}. If $\bar N>V(\boldsymbol{\omega}^{k_0})/S$, then the inequality~\eqref{eqn:contr_conv_zero_wMPC} implies $V(\boldsymbol{\omega}^{k_{\bar N}})<0$, contradicting the positive definiteness of $V$. Hence, Algorithm~\ref{algo:ho-dclf} generates a control sequence which steers $x(k)$ to zero as $k$ goes to infinity. 
\end{proof}

\begin{remark}\label{rkm:relxkny}
{\em
The sequence $\{(x(k_n), \boldu^{k_n}_{[0:q-1]})\}_{n=0}^\infty$, defined in the proof of Theorem~\ref{thm:algo_convtozero}, is related to $y(k)$, which was introduced earlier to extend~\eqref{eqn:flexstepdynamics}: The augmented variable $y(k)$ is basically equal to $(x(k_n), \boldu^{k_n}_{[0:q-1]})$, except that the elements between the $k_{n+1}$-th and $k_n$-th element are filled up with the constant $(x(k_n), \boldu^{k_n}_{[0:q-1]})$ for all $n \in \mathbb{N}$, i.e.
\begin{align*}
    &[\underbrace{y(k_0), \dots, y(k_1-1)}_{\ell^{k_0}_{decr} \text{ elements}}, \underbrace{y(k_1), \dots, y(k_2-1)}_{\ld^{k_1} \text{ elements}}, y(k_2), \dots]=\\ 
    &\quad [\underbrace{(x(k_0),\boldu^{k_0}_{[0:q-1]}),\dots, (x(k_0),\boldu^{k_0}_{[0:q-1]})}_{\ell^{k_0}_{decr} \text{ times}},\underbrace{(x(k_1),\boldu^{k_1}_{[0:q-1]}), \dots, (x(k_1),\boldu^{k_1}_{[0:q-1]})}_{\ld^{k_1} \text{ times}},(x(k_2),\boldu^{k_2}_{[0:q-1]}), \dots], 
\end{align*}
where $\ell_{decr}^{k_0}$ is the index $\ell_{decr}$ found in Algorithm~\ref{algo:ho-dclf} through the first optimization instance, occurring at $k_0$, $\ell_{decr}^{k_1}$ is found through the next optimization instance, occurring at $k_1$, and so on. \oprocend
}
\end{remark}

\begin{proof}[Proof of Theorem~\ref{thm:algo_conv}]
By solving Problem~\ref{prob:ho-dclf} repeatedly using Algorithm~\ref{algo:ho-dclf}, we generate a control sequence which, by Theorem~\ref{thm:algo_convtozero}, steers any state $x(k_0)$ in the set $X$ to zero. This proves the first statement. For the second statement we show that for all $\varepsilon>0$ and any $k_0 \in \mathbb{N}$, there exists $\bar \delta>0$ such that for any state $x(\initimetimev)\in X$ with $0 <\|x(\initimetimev)\|<\bar \delta$, the state governed by~\eqref{eqn:flexstepdynamics}
satisfies $\|x(k)\|<\varepsilon$, for all $k\geq k_0$. To this end, fix $\varepsilon> 0$ and $k_0\in \mathbb{N}$. As before, we choose $\mathbf{p} \in \mathbb{R}^n \times \mathbb{R}^{p, q}$ such that there exists $r \in (0,\varepsilon)$ satisfying equation~\eqref{eqn:rad_unbdd_proof} and define $B_r, \beta$ and $\Omega_\beta$ accordingly.
By using again the continuity of $V$ and Assumption~\ref{assump:small_control_prop}, we deduce that there exists $\bar\delta>0$ such that for any $x(\initimetimev) \in X$ with $0<\|x(\initimetimev)\|<\bar\delta$, there exists $\boldu^{k_0}_{[0:q-1]}\in \mathbf{U}_{[0:q-1]}(x(k_0))$ with which~\eqref{eqn:Vlessminbeta} holds, i.e.
\begin{align*}
    V(x(k_0),\boldu^{k_0}_{[0:q-1]})< \min\{0.5\beta \sigma_1,0.5\beta \sigma_2,\beta\}.    
\end{align*}
We claim that with the bound~\eqref{eqn:Vlessminbeta} at hand,  Algorithm~\ref{algo:ho-dclf} produces control inputs such that $V(x(k), \mathbf{u}^{k}_{[0:q-1]}) \leq \beta$ for all $k \geq k_0$. Together with equation~\eqref{eqn:rad_unbdd_proof}, this would imply that $(x(k),\boldu^{k}_{[0:q-1]})$ stays in $\Omega_\beta$ for $k\geq k_0$. We first consider the case $\sigma_1 \leq \sigma_2$ to prove the claim. Clearly, $V(x(k_0), \mathbf{u}^{k_0}_{[0:q-1]}) < 0.5\beta\sigma_1$. Consider $x(k_0)$ as the current state in the flexible-step MPC scheme, which defines the initial state of Problem~\ref{prob:ho-dclf}, $x^0=x(k_0)$, and let ${\mathbf{u}^{*}_{[0:N-1]}}(x(k_0))$ with $N=\max\{q+m,N_p\}$ be the optimal input that is obtained in step three of Algorithm~\ref{algo:ho-dclf}. Using the constraint \adc~of Problem~\ref{prob:ho-dclf}, we have  
\begin{align*}
    \sigma_2V(x^{2},\boldu^*&_{[2:q+1]}(x(k_0))) + \sigma_1 V(x^{1},\boldu^*_{[1:q]}(x(k_0))) < 2V(x(k_0),\boldu^{k_0}_{[0:q-1]}) < 2 \frac{\beta\sigma_1}{2} = \beta\sigma_1,
\end{align*}
and therefore, $V(x^{1},\boldu^*_{[1:q]}(x(k_0))) < \beta$ and in an analogous manner $V(x^{2},\boldu^*_{[2:q+1]}(x(k_0)))$ $<\beta$.
The case $\sigma_1 > \sigma_2$ follows similarly. By Proposition~\ref{thm:ex_subsqn_stability}, either $V(x^{1},\boldu^*_{[1:q]}(x(k_0)))$ or $V(x^{2},\boldu^*_{[2:q+1]}(x(k_0)))$ is less than $V(x^{k_0}, \boldu^{k_0}_{[0:q-1]})$ and determines $\ld\in \{1,2\}$. If $\ld=1$, Algorithm~\ref{algo:ho-dclf} has found $x(k_0+1):=x^{1}$ and $\boldu^{k_0+1}_{[0:q-1]}:=\boldu^*_{[1:q]}(x(k_0))$, if $\ld=2$ the algorithm has also found $x(k_0+2):=x^{2}$ and $\boldu^{k_0+2}_{[0:q-1]}:=\boldu^*_{[2:q+1]}(x(k_0))$. The control input $\boldu^*_{[0:\ld-1]}(x(k_0))$ is implemented and the above argument can be repeated with the new state control pair $(x(k_0+l_{decr}),\boldu^{k_0+l_{decr}}_{[0:q-1]})$. The repetition of this process ultimately proves the claim. Therefore, 
\[
(x(k), \boldu^k_{[0:q-1]}) \in \Omega_\beta \Rightarrow (x(k), \boldu^k_{[0:q-1]}) \in B_r \quad \forall k \geq k_0,
\]
and hence, we obtain $\|x(k)\| \leq \|(x(k),\boldu^{k}_{[0:q-1]})\| \leq r < \varepsilon$ for all $k\geq k_0$. As a result, if $x(\initimetimev) \in X$ satisfies $0<\|x(\initimetimev)\|<\bar \delta$, then $\|x(k)\|<\varepsilon$ for all $k \geq k_0$.
\end{proof}

\section{Proof of Proposition~\ref{thm:1st_system_fails_brockett}}\label{sec:appendixB}

In this appendix, we prove Proposition~\ref{thm:1st_system_fails_brockett}. Let us first recall Brockett's necessary condition for the discrete-time case \cite{WL-CB:94}.
\begin{theorem}\label{thm:brockett_cond}
Consider a discrete-time nonlinear control system governed by 
\begin{align*}
    x^{k+1} = f(x^k,u^k),
\end{align*}
with $f(0,0)=0$ and $f$ being smooth in a neighborhood of $(0,0)$. A necessary condition for the existence of a smooth state feedback control law $u^k = u(x^k)$ which renders $(0,0)$ locally asymptotically stable is that for the mapping 
\begin{align}\label{eqn:nec-cond-brockett-discrete}
    \varphi: (x,u) &\mapsto f(x,u)-x\nonumber
    \intertext{the equation}
    \varphi(x,u)&=y
\end{align}
shall be solvable for all sufficiently small $y$. 
\end{theorem}
In other words, if we can find a sufficiently small $y$ such that \eqref{eqn:nec-cond-brockett-discrete} is not solvable, then Brockett's necessary condition is violated. This is precisely what we will do in the next proof.
\begin{proof}[Proof of Proposition~\ref{thm:1st_system_fails_brockett}]
To show this result, we will apply Theorem~\ref{thm:brockett_cond} to system~\eqref{sys:plus_000abs}:
\begin{align*}
    \varphi(x,u) &= x + h
        \begin{bmatrix}
        1\\ 
        0\\
        -x_2\\ 
        x_3
        \end{bmatrix}
        u_1 + h
        \begin{bmatrix}
        0\\ 
        1\\
        x_1\\
        x_2
        \end{bmatrix}
        u_2+h
        \begin{bmatrix}
        0\\ 
        0\\
        0\\
        |x_4|
        \end{bmatrix} -x
        = h
        \begin{bmatrix}
        1\\ 
        0\\
        -x_2\\ 
        x_3
        \end{bmatrix}
        u_1 + h
        \begin{bmatrix}
        0\\ 
        1\\
        x_1\\
        x_2
        \end{bmatrix}
        u_2 +h
        \begin{bmatrix}
        0\\ 
        0\\
        0\\
        |x_4|
        \end{bmatrix}\hspace{-0.1cm}.
\intertext{For all $\varepsilon >0$ the equation $\varphi(x,u) = [0\ 0\ 0\ -\hspace{-0.05cm}y_4]^T$ has no solution for any $0<y_4<\varepsilon$. To see this, let us write out the above equation}
\begin{bmatrix}
        0\\
        0\\
        0\\
        -y_4
\end{bmatrix} &=h\left(
        \begin{bmatrix}
        1\\ 
        0\\
        -x_2\\ 
        x_3
        \end{bmatrix}
        u_1 +
        \begin{bmatrix}
        0\\ 
        1\\
        x_1\\
        x_2
        \end{bmatrix}
        u_2
        +
        \begin{bmatrix}
        0\\ 
        0\\
        0\\
        |x_4|
        \end{bmatrix}\right)
    = h\begin{bmatrix}
        u_1 \\u_2\\-x_2u_1 + x_1u_2\\x_3u_1 + x_2u_2 + |x_4|
    \end{bmatrix}\hspace{-0.1cm}.
    \intertext{The first and second component of the equation imply that the controls need to be zero. If the controls are zero, the right-hand side of the fourth component becomes}
    h(x_3u_1 &+ x_2u_2 + |x_4|) = h(0 + 0 + |x_4|).
\end{align*}
Now this shall be equal to the left-hand side of the fourth component, i.e.
    $-y_4 = h|x_4|.$
This equation demands that $-y_4$, which is negative, shall equal $h|x_4|$, which is  non-negative, thus it is not solvable. Brockett's necessary condition is not satisfied for the system \eqref{sys:plus_000abs}. 
\end{proof}

\end{document}